\documentclass[12pt, openany, reqno]{amsart}
\date{20 May, 2023}\usepackage[T1]{fontenc}
\usepackage[utf8]{inputenc}
\usepackage[english]{babel}
\usepackage{amsmath}
\usepackage{amssymb}
\usepackage{graphicx}
\usepackage{amsfonts}
\usepackage{amssymb}
\usepackage{mathtools}
\usepackage{graphicx}
\usepackage{subcaption}
\usepackage{float}
\usepackage{hyperref}
\usepackage{epstopdf}
\usepackage{fancyhdr}
\usepackage{url}
\usepackage{enumerate}
\usepackage{booktabs} 
\usepackage{xcolor}
\usepackage{bm}
\usepackage{tikz-cd}
\usetikzlibrary{cd}
\usepackage[mathscr]{euscript}
\usepackage{enumitem}
\usepackage{setspace}
\usepackage[normalem]{ulem}

\usepackage[a4paper,includefoot,margin=2cm]{geometry}

\def\CC{\ensuremath{\mathbb{C}}}
\def\RR{\ensuremath{\mathbb{R}}}

\def\cifty{\ensuremath{C^{\infty}}}

\def\PP{\ensuremath{\mathbb{P}}}
\def\cp{\ensuremath{\CC\PP^{1}}}

\def\HH{\ensuremath{\mathbb{H}}}

\newcommand{\forms}[2]{\Omega^{#1}(#2)}
\newcommand{\secs}[1]{\forms{0}{#1}}

\newcommand{\conj}[1]{\overline{#1}}
\newcommand{\holosecs}[1]{\mathrm{H}^0{(#1)}}

\newcommand{\abs}[1]{\left|#1\right|}

\def \DD{\partial}
\def \DDz{\DD_{z}}
\def \DDDz{\DD_{\conj{z}}}
\def \DDD{\overline{\partial}}

\DeclareMathOperator{\End}{\ensuremath{End}}

\DeclareMathOperator{\rank}{\ensuremath{rank}}
\DeclareMathOperator{\psl}{\ensuremath{PSL}}
\DeclareMathOperator{\ssl}{\ensuremath{SL}}

\DeclareMathOperator{\Proj}{\ensuremath{Proj}}

\DeclareMathOperator{\hol}{\ensuremath{hol}}
\DeclareMathOperator{\ord}{\ensuremath{ord}}
\DeclareMathOperator{\Hom}{\ensuremath{Hom}}
\DeclareMathOperator{\Div}{\ensuremath{div}}
\DeclareMathOperator{\Gr}{\ensuremath{Gr}}

\def\sl{\ensuremath{\ssl(2,\CC)}}
\def\slr{\ensuremath{\ssl(2,\RR)}}

\def\beltrami{\ensuremath{\mathcal{B}}}

\theoremstyle{plain}
\newtheorem{theorem}{Theorem}
\newtheorem*{theorem*}{Theorem}
\newtheorem{lemma}[theorem]{Lemma}
\newtheorem{proposition}[theorem]{Proposition}
\newtheorem{corollary}[theorem]{Corollary}

\theoremstyle{remark}
\newtheorem{remark}[theorem]{Remark}
\newtheorem{example}[theorem]{Example}

\theoremstyle{definition}
\newtheorem{definition}[theorem]{Definition}

\numberwithin{theorem}{section}
\numberwithin{equation}{section}

\DeclareMathAlphabet{\mathpzc}{OT1}{pzc}{m}{it}

\title{The conformal limit and projective structures}

\author[P. M. Silva]{Pedro M.\ Silva}
\author[P.~B.\ Gothen]{Peter B.\ Gothen}

\address{Centro de Matemática da Universidade do Porto 
	\newline\indent Departamento de Matemática, Faculdade de Ci\^encias da Universidade do Porto 
	\newline\indent Rua do Campo Alegre s/n, 4169-007 Porto, Portugal}

\email{pmsilva@fc.up.pt  and pbgothen@fc.up.pt}

\thanks{First author supported by grant PD/BD/150349/2019 and both authors partially supported by CMUP (UIDB/00144/2020 and UIDP/00144/2020) and the project EXPL/MAT-PUR/1162/2021 all funded by FCT (Portugal) with national funds.  }

\begin{document}

	\begin{abstract}
		The non-abelian Hodge correspondence maps a polystable $\slr$-Higgs bundle on a compact Riemann surface $X$ of genus $g\geq2$ to a connection which, in some cases, is the holonomy of a branched hyperbolic structure. On the other hand, Gaiotto's conformal limit maps the same bundle to a partial oper, i.e., to a connection whose holonomy is that of a branched complex projective structure compatible with $X$.
		In this article, we show how these are both instances of the same phenomenon: the family of connections appearing in the conformal limit can be understood as a family of complex projective structures, deforming the hyperbolic ones into the ones compatible with $X$.
		We also show that, when the Higgs bundle has zero Toledo invariant, this deformation is optimal, inducing a geodesic on Teichmüller's metric space.
	\end{abstract}
	
	\maketitle
	\tableofcontents
	\onehalfspacing
	\section{Introduction}
	
	The geometric study of representations $\rho:\pi_1(M)\to G$ of the fundamental group of a surface $M$ into a Lie group $G$ is an active topic of research. It has a long and deep-rooted history starting perhaps with the uniformization theorem of a compact Riemann surface by a Fuchsian group and it is intimately connected with Teichmüller's work on the space $\mathscr{T}(M)$ that now bears his name. The particular approach to the subject we will follow relies on using the non-abelian Hodge correspondence to construct geometric structures, which are naturally associated to $\rho,$ in order to better understand the nature of the representation. 
	
	The use of Higgs bundles to understand geometric structures on surfaces originates with Hitchin's parametrization of Teichmüller spaces in his seminal paper \cite{hitchin:1987vg}. More recently several authors have used Higgs bundles to study other geometric structures; see, e.g., \cite{baraglia:2010,biswas:2021aa,collier:2020,collier:2023,labourie:2007} and, for the case of complex projective structures \cite{alessandrini:2019aa,alessandrini-etal:2021}.
	
	Let $X$ be a closed Riemann surface of genus $g\geq 2$ with underlying smooth surface $M$. We use Higgs bundles to construct a family of complex projective structures which interpolates between the hyperbolic structure corresponding to a quadratic differential $q$ on $X$ in Hitchin's parametrisation of $\mathscr{T}(M)$, and the complex projective structure corresponding to the $\hbar$-conformal limit of the Higgs bundle given by $q$ \cite{gaiotto:2014td,dumitrescu:2021uz,collier:2018aa}. This family is parametrized by $\hbar\in\CC^*$ and $R>0$ such that $\abs{\hbar R}\leq 1$. In fact, the construction works more generally for the branched hyperbolic structures studied in \cite{biswas:2021aa} and produces branched projective structures.
	
	We now explain our construction in more detail. Let $G$ be a reductive Lie group (real or complex). The non-abelian Hodge correspondence establishes a homeomorphism between the moduli space $\mathcal{M}^{G}_{\mathrm{Dol}}$ of polystable $G$-Higgs bundles on $X$ and the moduli space $\mathcal{M}^{G}_{\mathrm{dR}}$ of flat reductive connections; the latter can in turn be identified with the moduli space $\mathcal{M}^{G}_{\mathrm{B}}$ of semisimple representations of $\pi_1(M)$ in $G$ via the holonomy representation of a flat connection.
	
	In this paper, we focus on the rank two case. An $\sl$-Higgs bundle is a pair $(E,\Phi)$, where $E$ is a rank 2 holomorphic vector bundle on $X$ with trivial determinant bundle and $\Phi\in \holosecs{\End_0(E) \otimes K}$ is a traceless endomorphism valued holomorphic 1-form (we have written $K$ for the holomorphic cotangent bundle, which coincides with the canonical line bundle of $X$). An $\slr$-Higgs bundle can be viewed as an $\sl$-Higgs bundle of the form
	\begin{displaymath}
		\left(E=L \oplus L^{-1},\Phi = 
		\begin{pmatrix}
			0 & \alpha \\
			\beta & 0
		\end{pmatrix}
		\right),
	\end{displaymath}
	where $L$ is a holomorphic line bundle; note that $\alpha\in H^0(L^2K)$ and $\beta\in H^0(L^{-2}K)$. The topological invariant $d=\deg(L)$ gives a partition of $\mathcal{M}^{\slr}$ into subspaces $\mathcal{M}^{\slr}_d$. The invariant $\deg(L)$ is bounded by the Milnor--Wood inequality $\abs{\deg(L)}\leq g-1$ and, by symmetry, $\mathcal{M}^{\slr}_d \cong \mathcal{M}^{\slr}_{-d}$, so we may assume $d\geq 0$. The subspaces $\mathcal{M}^{\slr}_d$ are connected components, except $\mathcal{M}^{\slr}_{g-1}$. The latter has $2^{2g}$ components which are identified when passing to the quotient $\mathrm{PSL}(2,\RR)$; we refer to any such maximal component as a Hitchin component. Since for $d>0$ the polystability condition is equivalent to $\beta\neq 0$, one has $L^2\cong K$ for $d=g-1$, the square roots of $K$ thus accounting for the $2^{2g}$ Hitchin components. Moreover, each component is parameterized by quadratic differentials $q=\alpha\in\holosecs{K^2}$. Hitchin \cite{hitchin:1987vg}
	proved that a Hitchin component parameterizes all hyperbolic structures on $M$ and that under the non-abelian Hodge correspondence a hyperbolic structure is sent to its holonomy. Biswas et al.~\cite{biswas:2021aa} generalized this to show that any $\slr$-Higgs bundle with $\Div(\alpha)\geq\Div(\beta)$ gives rise to a branched hyperbolic structure with branching divisor $\Div(\beta)$.
	
	On the other hand, Gaiotto's conformal limit \cite{gaiotto:2014td}, is related to complex projective structures. For each Higgs bundle it introduces a family $\nabla_{\hbar,R}$ of flat connections parameterised by $\hbar\in\CC^*$ and $R>0$ such that $\nabla_{1,1}$ is the flat connection corresponding to the Higgs bundle under the non-abelian Hodge correspondence. For Higgs bundles in an $\slr$-Hitchin component, Gaiotto conjectured that the \emph{conformal limit} $\nabla_{\hbar,0} = \lim_{R\to 0}\nabla_{\hbar,R}$ exists and is an oper, i.e., defines a complex projective structure compatible with the Riemann surface structure of $X$. The conjecture has been proven (for Hitchin components for any split real $G$) by Dumitrescu et al.~\cite{dumitrescu:2021uz} and (for any Higgs bundle with stable $\CC^*$-limit at zero) by Collier--Wentworth \cite{collier:2018aa}. In particular, for $\beta\neq 0$ the conformal limit for an $\sl$-Higgs bundle is a partial oper which defines a branched projective structure.
	
        In the present specific case of $G=\slr$, we take advantage of the symmetry of Hitchin's equations to identify the conformal limit explicitly (see Theorem~\ref{conflimit} below) without using the Implicit Function Theorem in Banach Spaces, in contrast to the references just cited.

	Our main results on projective structures can now be summarised as follows.
	
	\begin{theorem*}
		Let 
		$(E=L\oplus L^{-1}, \Phi=\begin{psmallmatrix}
			0 & \alpha \\
			\beta & 0
		\end{psmallmatrix})$ 
		be an $\slr$-Higgs-bundle with $0\leq\deg(L)\leq g-1$ and $\beta\neq 0$. 
		Assume that $\abs{\hbar R}\leq 1$ (or $\abs{\hbar R}< 1$ when $\deg(L)=0$) and $\Div(\alpha)\geq\Div(\beta)$. Then the following results hold.
		\begin{enumerate}
			\item There is a Riemann surface structure  $X_\mu\in\mathscr{T}(M)$ associated to a Beltrami differential $\mu=\mu(\hbar,R)$ on $X$ and a branched projective structure $\mathcal{P}(\hbar,R)\in\mathcal{B}(M)$ with branching divisor $\Div(\beta)$, compatible with $X_\mu$. 
			\item The family $\mathcal{P}(\hbar,R)$ depends continuously on $(\hbar,R)$ and when $\deg(L)>0$ it interpolates between a branched hyperbolic structure $\mathcal{P}(\hbar,1)$ and the branched projective structure given by the partial oper $(\nabla_{\hbar,0},L)$.
			\item For $\deg(L)=0$ the curve $R\mapsto X_{\mu(\hbar,R)}$ in $\mathscr{T}(M)$ is a geodesic ray in the Teichmüller metric.
		\end{enumerate}
		Both $\mu(\hbar,R)$ and $\mathcal{P}(\hbar,R)$ are defined in terms of the Higgs bundle data (see Theorem \ref{main theorem}) and, in particular, the holonomy of $\mathcal{P}(\hbar,R)$ is that of $\nabla_{\hbar, R}$.
	\end{theorem*}
	
	Statement (1) of the above theorem is given in Theorems~\ref{main theorem} (in the case $\deg(L)>0$) and \ref{main theorem degree zero} (in the case $\deg(L)>0$). The continuity statement (2) is discussed in Section~\ref{section curves} (cf.~\ref{continuity on configuration space}). Statement (3) is given in Theorem~\ref{thm:teichmuller-geodesic}.
	
	The organization and results of the paper are as follows.
	We begin by recalling the non-abelian Hodge correspondence and the conformal limit in Section~\ref{section conformal limit}. We give a new proof of the existence of the conformal limit for any polystable $\slr$-Higgs bundle with non-zero Toledo invariant in Section~\ref{section conformal limit existence}, using the simple
	Proposition~\ref{change solution} regarding the symmetry of Hitchin's equation.
	Section~\ref{expository} is expository, collecting some known facts, which we write as tailored by our needs.
	We recall the definitions of complex projective structure in Section~\ref{section projective structures} and of the branched case in Section~\ref{section branched case}. We state Gunning's transversality criterion and its relation with (partial) opers in Sections~\ref{section Gunning} and \ref{section partial oper}. We finally prove the conformal limit to be a partial oper when the Toledo invariant is non-zero in Section~\ref{section conformal limit is oper}. 
	In Section~\ref{section beltrami differentials} we recall some facts about Beltrami differentials and classical Teichmüller theory. Section~\ref{section branched projective structures coming from the limit} presents the main results, showing that, under some conditions determined by Brill-Noether considerations, the family $\nabla_{\hbar, R}$ appearing in the conformal limit is the holonomy of a branched projective structure which is compatible with a Riemann surface structure $X_\mu$ determined by a Beltrami differential $\mu(\hbar, R)$ which we explicitly calculate. 
	The case of zero Toledo invariant is approached in Section~\ref{degree zero section}, where we restate and prove similar results, namely that the conformal limit exists and it is a partial oper (Section~\ref{section the conformal limit degree zero}) and that, under similar conditions, the connections $\nabla_{\hbar, R}$ are associated to branched projective structures again compatible with $X_{\mu(\hbar,R)}$, for some explicit Beltrami differential $\mu(\hbar,R)$ (Section~\ref{section projective structures degree zero}).
	Finally, we provide some geometric understanding of the results in Section~\ref{section geometric interpretation of the results}. We show, in Section~\ref{section teichmuller geodesics} that in the case of zero Toledo invariant the curve $R \to X_{\mu(\hbar,R)}$, for fixed $\hbar$, is a (reparametrization) of a geodesic for the Teichmüller metric in $\mathscr{T}(M)$. In Section~\ref{sections reality properties}  we also give conditions on $\hbar$ and $R$ under which the constructions produce branched hyperbolic structures.
	
	\subsection*{Acknowledgements}
	We would like to thank Q. Li for useful discussions and the
        referee for a careful reading of the manuscript.
	
	\section{The conformal limit} \label{section conformal limit}
	
	\subsection{Preliminaries}
	In this section, we introduce the conformal limit following \cite{dumitrescu:2021uz} and \cite{collier:2018aa}.
	We consider $X$ a closed Riemann surface of genus $g\geq2$ and $E$ a fixed (necessarily trivial) $\ssl(n,\CC)$-vector bundle of rank $n$ over $X$. We view an $\ssl(n,\CC)$-Higgs bundle as a pair $(\DDD_E, \Phi)$ where $\DDD_E$ is a $\DDD$-operator defining a holomorphic structure on $E$, which has trivialized determinant, and the Higgs field $\Phi\in\forms{1,0}{\End_0(E)}$ is a $\DDD_E$-holomorphic $(1,0)$-form with values in the traceless endomorphism bundle $\End_0(E)$. 
	\begin{definition}
		Let $(\DDD_E, \Phi)$ be a $\ssl(n,\CC)$-Higgs bundle. Let $H$ be a Hermitian metric on $E$ and let $F_{A_H}$ be the curvature of the Chern connection $A_H$ determined by $H$ and $\DDD_E$. Fix $R\in\RR^+$. A metric which induces the trivial metric on $\det(E)$ is called \emph{harmonic with parameter $R$} if it satisfies the $R$-scaled version of Hitchin's equation:
		\begin{equation} \label{Hitchin eq}
			F_{A_{H}} + R^2\,[\Phi,\Phi^{*_H}]=0.
		\end{equation}
	\end{definition}
	\begin{remark}
		If $R=1$ we get the usual Hitchin equation for the Higgs bundle $(\DDD_E,\Phi)$, and a metric which is harmonic with parameter $1$ is simply called \emph{harmonic}. 
		Thus a metric is harmonic with parameter $R$ if and only if it is a solution to Hitchin's equation for the Higgs bundle $(\DDD_E,R\Phi)$.
	\end{remark} 
	The existence of such a metric when certain stability conditions on $(\DDD_E, \Phi)$ are met is a part of the celebrated non-abelian Hodge correspondence.
	Recall that a $\ssl(n,\CC)$-Higgs bundle $(\DDD_E,\Phi)$ is \emph{(semi)stable} if every $\Phi$-invariant holomorphic subbundle $L$ of $E$ has (non-positive) negative degree. Further, it is called \emph{polystable} if it is a direct sum of stable Higgs bundles (of degree $0$). 
	\begin{theorem} 
		A $\ssl(n,\CC)$-Higgs bundle $(\DDD_E,\Phi)$ is polystable if and only if, for each $R\in\RR^+$, there is a harmonic metric with parameter $R$. This metric is unique if $(\DDD_E,\Phi)$ is stable.
	\end{theorem}
	The existence of such a metric allows one to use $R$ to deform the flat connection associated to the Higgs Bundle. 
	Given an $\ssl(n,\CC)$-Higgs bundle $(\DDD_E,\Phi)$ and a fixed $\hbar\in \CC^*$ one has the $\RR^+$-family of flat connections
	\begin{equation}  \label{family of flat}
		\nabla_{\hbar,R}=A_{H_R} + \hbar^{-1} \Phi + \hbar R^2 \Phi^{*_{H_R}}
	\end{equation}
	where $H_R$ is the harmonic metric with parameter $R\in \RR^+$ for $(\DDD_E,\Phi)$, $A_{H_R}$ is the Chern connection for $\DDD_E$ and $H_R$, and the adjoint $*_{H_R}$ is taken also with respect to the metric $H_R$.
	
	\begin{definition}
		The \emph{$\hbar$-conformal limit} of $(\DDD_E,\Phi)$ is the connection
		\begin{equation}
			\nabla_{\hbar,0}:=\lim_{R\to 0} \nabla_{\hbar,R}
		\end{equation}
		when it exists.
	\end{definition}
	\begin{remark}
		In the case $\hbar=1$ note that, if $R=1$, the connection $\nabla_{1,1}$ is just the one given by the usual non-abelian Hodge correspondence. 
	\end{remark}
	The existence of this limit was established in \cite{dumitrescu:2021uz,collier:2018aa} for the stable case. In the case of $\slr$-Higgs bundles, we shall give a more direct argument both in the stable case (in Section~\ref{section conformal limit existence}) and in the polystable case (in Section~\ref{section the conformal limit degree zero}). We do this by taking advantage of the fact that the structure group $\CC^*$ is abelian for $\slr$-Higgs bundles.
	\subsection{Explicit limit in the $\slr$ case} 
	\label{section conformal limit existence}
	Recall that we can view an $\slr$-Higgs bundle as an $\sl$-Higgs bundle $E$ with a holomorphic decomposition $E=L\oplus L^{-1}$, where $L$ is a holomorphic line bundle. Further, the Higgs field has the form $\Phi=\begin{psmallmatrix}
		0 & \alpha \\
		\beta & 0
	\end{psmallmatrix}$ for this decomposition, where $\alpha \in \holosecs{L^2 K}$ and $\beta \in \holosecs{L^{-2} K}$. For such a Higgs bundle the \emph{Toledo invariant} $\deg(L)$ satisfies a Milnor-Wood type inequality $0\leq |\deg(L)|\leq g-1$. 
	We can assume, by duality, that an $\slr$-Higgs bundle has $\deg(L)\geq0$.
	
	The $\sl$-Higgs bundle stability condition for the case $\deg(L)>0$ is simply $\beta\neq0$. This is because $L$ is the maximal destabilizing subbundle of $E$, which is not preserved by the Higgs field if and only if $\beta \neq 0$. There are no strictly polystable cases. For the case $\deg(L)=0$ the condition is more delicate, and it is analyzed in Section~\ref{degree zero section}. 
	
	The harmonic metric $H$ in either case is known to diagonalize (\cite{hitchin:1987vg} or \cite[Proposition 5.2]{alessandrini:2019aa}, for example) 
	with respect to this decomposition, so $H=\begin{psmallmatrix}
		h & 0\\0 &h^{-1}
	\end{psmallmatrix}$, where $h$ is a metric in the line bundle $L$. 
	
	If we choose a holomorphic frame for $L$, and the induced holomorphic frame in $E$, $h$ is locally given by a positive function, still denoted by $h$, or $h(z)$ if we want to make explicit the dependence on a complex coordinate $z$ in $X$. The Chern connection for $H$ is given in this frame by $A_{H}=d + \begin{psmallmatrix}	
		\DD\log h & 0\\
		0 &- \DD\log h
	\end{psmallmatrix}=\begin{psmallmatrix}	
		\DDz \log h \,dz & 0\\
		0 &- \DDz \log h \,dz
	\end{psmallmatrix}$.
	Further $\alpha$ and $\beta$ are given by 1-forms $\alpha = \alpha(z) dz$ and $\beta =\beta(z) dz$. Recalling that locally $\Phi^{*_H}=H^{-1}\conj{\Phi}^T H$, Hitchin's equations (\ref{Hitchin eq}) for this case $(R=1)$ read
	$$\begin{aligned}
		0&=F_{A_H} +[\Phi,\Phi^{*_H}]\Leftrightarrow \\
		0&= \begin{pmatrix}	
			\DDD	\DD\log h & 0\\
			0 &- \DDD \DD\log h
		\end{pmatrix} + \begin{pmatrix}
			0 & \alpha \\
			\beta & 0
		\end{pmatrix} \wedge \begin{pmatrix}
			0 & \conj{\beta} h^{-2} \\
			\conj{\alpha}h^2 & 0
		\end{pmatrix} + \begin{pmatrix}
			0 & \conj{\beta} h^{-2} \\
			\conj{\alpha}h^2 & 0
		\end{pmatrix} \wedge \begin{pmatrix}
			0 & \alpha \\
			\beta & 0
		\end{pmatrix}.\\
	\end{aligned}$$
	
	This simplifies to the single scalar (and unscaled) vortex equation
	\begin{equation*}
		\DDDz \DDz \log h = |\alpha|^2h^2 -|\beta|^2h^{-2}.
	\end{equation*}
	The $R$-scaled version is given in the following definition and just says that $H_R=\begin{psmallmatrix}
		h_R & 0\\0 &h_R^{-1}
	\end{psmallmatrix}$ is harmonic with parameter $R$ for $\left(E=L\oplus L^{-1}, \Phi=\begin{psmallmatrix}
		0 & \alpha \\
		\beta & 0
	\end{psmallmatrix}\right)$.
	\begin{definition} \label{vortex eqs}
		Let $L$ be a holomorphic line bundle and consider the sections $\alpha \in \holosecs{L^2 K}$ and $\beta \in \holosecs{L^{-2} K}$, and $R\in\RR^+$.
		A metric $h_R$ in $L$ \emph{solves the $R$-scaled vortex equation for $(\alpha, \beta)$} if locally in a holomorphic frame
		\begin{equation*}
			\DDDz \DDz \log h_R = R^2\left(|\alpha|^2h_R^2 -|\beta|^2h_R^{-2}\right).
		\end{equation*}
	\end{definition}
	\begin{remark} \label{zero curvature}
		For $R=0$ this is an equation for a metric of zero curvature. If such a Hermitian metric exists on $L$ then $\deg(L)=0$.
	\end{remark}
	\begin{example} \label{example hitchin}
		If the Higgs bundle lies in a Hitchin component then $$\left(E=K^{1/2}\oplus K^{-1/2}, 
		\Phi=\begin{psmallmatrix}
			0 & q \\
			1 & 0
		\end{psmallmatrix}\right),$$
		for a choice $K^{1/2}$ of square root of the canonical bundle $K$, from the $2^{2g}$ available, and $q\in \holosecs{K^{2}}$ a quadratic differential. 
		The $R$-scaled vortex equation is then just the $R$-scaled version of the abelian vortex equation as in Hitchin~\cite{hitchin:1987vg}
		\begin{equation*}
			\DDDz \DDz \log h = R^2 \left( |q|^2h^2 -h^{-2}\right).
		\end{equation*}
		In this case $g=h^{-2}$ is a metric in $(K^{1/2})^{-2}\cong K^{-1}\cong TX$, which satisfies
		\begin{equation}
			\DDDz \DDz \log g = 2 R^2 g\left(1-\frac{q\conj{q}}{g^2}\right).
		\end{equation}
		For $R=1$ and $q=0$, this is the equation for a Riemannian metric $g_0$ of constant negative curvature $-4$. 
	\end{example}
	In the case of $\slr$-Higgs bundles, the family of connections (\ref{family of flat}) that comes up in the conformal limit is thus
	\begin{align}
		\label{R family}
		\nabla_{\hbar,R}&=A_{H_R} + \hbar^{-1} \Phi + \hbar R^2 \Phi^{*_{H_R}} \notag\\
		&=d +  \begin{pmatrix}	
			\DD\log h_R & 0\\
			0 &- \DD\log h_R
		\end{pmatrix} + \hbar^{-1} 
		\begin{pmatrix}	
			0 & \alpha\\
			\beta &0
		\end{pmatrix} + \hbar R^2
		\begin{pmatrix}	
			0 & \conj{\beta} h_R^{-2} \notag\\
			\conj{\alpha}h_R^2 & 0
		\end{pmatrix}\\
		&=d+\begin{pmatrix}	
			\DD\log h_R & \hbar^{-1}\alpha + \hbar R^2 \conj{\beta} h_R^{-2}\\
			\hbar^{-1} \beta + \hbar R^2 \conj{\alpha}h_R^2  &- \DD\log h_R
		\end{pmatrix}.
	\end{align}
	
	The calculations \cite{dumitrescu:2021uz,collier:2018aa} of the conformal limit use the inverse function theorem in infinite dimensional Banach spaces to find suitable gauge transformations acting on $\nabla_{\hbar,R}$.  As we shall see next, in the case of $\slr$-Higgs bundles the calculation can be done directly, using that the moduli spaces of Higgs bundles and of flat connections are homeomorphic.

	The argument relates the solutions of the scaled equation with solutions of the unscaled one.
	\begin{proposition} \label{change solution}
		Let $L$ be a holomorphic line bundle, $\alpha \in \holosecs{L^2 K}$, $\beta \in \holosecs{L^{-2} K}$, and $R\in\RR^+$.
		A metric $h_R$ is a solution of the $R$-scaled vortex equation for $(\alpha,\beta)$ if and only if $h:= \frac{h_R}{R} $ is a solution of the unscaled vortex equation for $(R^2\alpha, \beta)$.
	\end{proposition}
	\begin{proof}
		Let $h:=R^{-1} h_R$ be a solution of the unscaled vortex equation for $(R^2\alpha, \beta)$. This happens if and only if
		\begin{equation*}
			\DDDz \DDz \log h = |R^2\alpha|^2h^2 -|\beta|^2h^{-2}.
		\end{equation*}
		Given the fact that $R$ does not depend on $z$, i.e., $\DDz \log h=\DDz \log h_R-\DDz \log R=\DDz \log h_R$ , the expression is equivalent to
		\begin{equation*}
			\DDDz \DDz \log h_R= |R^2\alpha|^2R^{-2}h_R^2 - |\beta|^2 R^{2} h_R^{-2}=R^{2}\left(|\alpha|^2h_R^2 - |\beta|^2 h_R^{-2}\right),
		\end{equation*}
		and so $h_R$ solves the $R$-scaled vortex equation for $(\alpha,\beta)$.
	\end{proof}
	
	\begin{corollary}
		\label{limit exists in space of harmonic metrics}
		Let $\deg(L)>0$. Then the pointwise limit of $h:= \frac{h_R}{R} $ as
		$R\to 0$ exists as a metric and it is a solution of the unscaled vortex equation
		for $(0, \beta)$.
	\end{corollary}
	
	\begin{proof}
		By Propostion~\ref{change solution} $h$ is a solution of the
		unscaled vortex equations for $(R^2\alpha, \beta)$. This describes a
		continuous path $(E=L\oplus L^{-1},\Phi(R))$ of polystable Higgs
		bundles, with
		$\Phi(R)=\begin{psmallmatrix}
			0 & R^2 \alpha\\
			\beta & 0 \end{psmallmatrix}$.
		Note that it is a well-defined path since, $(E,\Phi(R))$ is stable
		for all $R$. This happens even for $R=0$, since then the Higgs bundle
		is stable because $\deg(L)>0$. (The subbundle $L$ is maximally
		destabilizing.) Under the homeomorphism to the space of harmonic
		bundles, this is mapped to a continuous path of metrics, and
		$\lim_{R\to 0} h$ is the metric associated to $\Phi(0)$.
	\end{proof}
	
	\begin{remark} \label{degree zero}
		If $h_R$ is a solution of the $R$-scaled vortex equation for a fixed
		$(\alpha, \beta)$, it was already noted in \cite{dumitrescu:2021uz}
		(for the Hitchin component) that $\lim_{R\to 0} h_R$ does not exist
		in general. In our case this can be seen directly from the fact that a solution
		for the $0$-scaled equation is just a metric with curvature equal to
		zero (cf.\ Remark~\ref{zero curvature}), and thus it can only exist
		in the case $\deg(L)=0$, but not for general $L$. More precisely,
		in view of the Corollary, it follows that $h=h_R/R$ tends to the solution of the
		unscaled vortex equations with $(0,\beta)$, and so $h_R=h \,R$
		tends to zero as fast as $R\to0$ for $\deg(L)>0$.
		We note also that for $\deg(L)=0$ the limit $\lim_{R\to 0} h$ does not exist, since in that case $(0,\beta)$ defines a non-polystable Higgs bundle (cf.\ Section~\ref{degree zero section} for further details).
	\end{remark}

	Using Corollary~\ref{limit exists in space of harmonic metrics} we can now calculate the conformal limit.
	
	\begin{theorem}\label{conflimit} \cite[Prop.~5.1]{collier:2018aa}
		Let $X$ be a closed Riemann surface of genus $g\geq 2$. Consider the vector bundle $E=L\oplus L^{-1}$, with $1\leq\deg(L)\leq g-1$, and induced holomorphic structure $\DDD_E$ and Higgs field $\Phi=\begin{psmallmatrix}
			0 & \alpha\\
			\beta &0
		\end{psmallmatrix}$, where $\alpha \in \holosecs{L^2K}$ and $0\neq \beta \in \holosecs{L^2K}$. 
		
		Then, the $\hbar$-conformal limit $\nabla_{\hbar,0}$ of the $\sl$-Higgs bundle $(\DDD_E, \Phi)$ exists.
		Using the holomorphic frame of $E$ induced by a holomorphic frame of $L$, the coordinate representation of $\nabla_{\hbar,0}$ is
		\begin{equation}\label{eq:conflimit}
			\nabla_{\hbar,0}=d+\begin{pmatrix}	
				\DD\log h_0 & \hbar^{-1}\alpha + \hbar \conj{\beta} h_0^{-2}\\
				\hbar^{-1} \beta  &- \DD\log h_0
			\end{pmatrix},
		\end{equation}
		where $h_0$ is the solution of the unscaled vortex equations for $(0,\beta)$.
	\end{theorem}
	
	\begin{proof}
		By Proposition~\ref{change solution}, for each $R\in \RR^+$, we can write $h_R$ as $h_R=h\, R$, where $h$ is the solution of the unscaled vortex equation for $(R^2 \alpha, \beta)$. By Corollary~\ref{limit exists in space of harmonic metrics}, the limit of $h$ as $R\to 0$ exists and we have $\lim_{R\to 0} h=h_0$. We also note that $\DD_z \log R=0$.
		
		We are now able to compute the following limits:
		\begin{align*}
			&\displaystyle \lim_{R\to 0} \DD_z \log h_R= \lim_{R\to 0} \DD_z \log \left(h \,R\right) =\lim_{R\to 0}\left( \DD_z \log h + \DD_z \log R \right)=\lim_{R\to 0}\DD_z \log h=\DD_z \log h_0,\\
			&\displaystyle \lim_{R\to 0}  R^2 h_R^2= \lim_{R\to 0} R^4 h=0, \\
			&\displaystyle \lim_{R\to 0}  R^2  h_R^{-2}= \lim_{R\to 0}  h^{-2}=h_0^{-2}.
		\end{align*}
		Taking $R\to 0$ in the family $\nabla_{\hbar, R}$ of (\ref{R family}) we then get the existence and explicit form of the conformal limit stated.
	\end{proof}
	
	\begin{remark} \label{continuity on configuration space}
		Note that for $R\neq0$, the family $\nabla_{\hbar,R}$ in \ref{R family}, as a function of the parameters $\hbar$ and $R$ into the configuration space of flat connections, is continuous. So, in fact, we have shown here that this continuity extends to $R=0$.
	\end{remark}
	
	\begin{remark} \label{components existence}
		The connected components of the moduli space of $\slr$-Higgs bundles are as follows \cite{hitchin:1987vg}. There are $2\cdot 2^{2g}$ Hitchin components corresponding to the maximal Toledo invariant $\abs{\deg(L)}=g-1$. Further, there are $2g-1$ non-maximal components, corresponding to $0\leq\abs{\deg(L)}<g$. The existence of the conformal limit is thus established for all the components except for the minimal one, i.e., the one for which $\deg(L)=0$. This will be done in Section~\ref{degree zero section}. 
	\end{remark}
	
	\begin{remark}
		Our argument can in fact be generalized to show the existence of the conformal limit in the configuration space for any polystable $\sl$-Higgs bundle. This will be treated elsewhere. Here we have limited ourselves to the $\slr$-case because our main interest lies in the connection with branched projective structures.
	\end{remark}
	
	\section{Branched projective structures and partial $\sl-$opers} \label{expository}

	In this section we recall some basic facts on branched projective structures and the closely related concept of partial $\sl-$opers. 
	
	\subsection{Projective structures} \label{section projective structures}
	A projective structure on a surface $M$ is given by a maximal atlas $\mathcal{A}=\{(U_\alpha,\phi_\alpha:U_\alpha\to\cp)\}$, whose charts are diffeomorphisms onto open sets of $\cp$ and whose transition functions are Möbius transformations. These are examples of locally homogeneous structures or $(G,X)$-manifolds, where $X=\cp$ and $G=\psl(2,\CC)$ (\cite[Chapter 3]{thurston:2022aa}).
	
	Equivalently, a projective structure is a pair $(\rho,d)$ where $\rho:\pi_1(M)\to \psl(2,\CC)$ is a representation, called the \emph{holonomy} of the structure, and $d:\widetilde{M}\to \cp$, the \emph{developing map}, is a $\rho$-equivariant local diffeomorphism with domain the universal cover $\widetilde{M}\to M$. The pair $(\rho,d)$ is uniquely defined up to the action of $G$ by conjugation on $\rho$ and by post-composition on $d$. 
	
	A third description of projective structures is obtained by constructing the flat $\cp$-bundle $\PP E_\rho\to M$ whose holonomy is given by $\rho$. The developing map is then canonically identified with a section $s$ of $\PP E_\rho$ and it is a local diffeomorphism if and only if this section is transverse to the flat connection on $\PP E_\rho$. The equivalence is now given by the usual action of the gauge group simultaneously on flat connections and sections.
	
	Each of these structures is then considered modulo diffeomorphisms which are locally in $\psl(2,\CC)$ and are isotopic to the identity, in a construction similar to Teichmüller space. Any incarnation of this space is denoted by $\mathscr{P}(M)$.
	
	We note that, since $\psl(2,\CC)$ acts holomorphically on $\cp$, every projective structure $\mathcal{P}$ seen as maximal atlas on a surface $M$, induces a Riemann surface structure $X$, because this maximal atlas is, in particular, an atlas for a complex manifold in the usual sense. It is of course possible that projective structures $\mathcal{P}$ and $\mathcal{P}'$ define the same complex structure $X$ even though they are not equivalent as projective structures. We then say that $\mathcal{P}$  (and $\mathcal{P}'$) are \emph{compatible} with $X$.
	One can rephrase this statement by saying that the projective structure $\mathcal{P}$ is a \emph{stiffening} of the fixed complex structure $X$ (\cite[page 112]{thurston:1997aa}). In this case, the local charts $\phi_\alpha:U_\alpha\to\cp$ are holomorphic as local functions in $X$ and can be written in holomorphic coordinates as the identity map $z \mapsto z$.
	Further, the universal cover $\widetilde{M}$ can be made into a Riemann surface $\widetilde{X}$.  The developing map of $\mathcal{P}$ becomes a local biholomorphism for $\widetilde{X}$. Also the total space $\PP E_\rho$ becomes a complex manifold and the section $s$ becomes a holomorphic section. 
	We denote by $\mathscr{P}(X)$ the space of projective structures which are a stiffening of $X$ modulo the same action of diffeomorphisms which are locally in $\psl(2,\CC)$ and are isotopic to the identity. 
	Observe that we have a map 
	\begin{equation*}
		\mathscr{P}(M) \to \mathscr{T}(M)
	\end{equation*}
	sending a projective structure $\mathscr{P}$ to its induced Riemann surface structure $X$, and for which the fiber over $X$ is $\mathscr{P}(X)$.

	\subsection{The branched case} \label{section branched case}
	
	We are going to describe branched projective structures following Mandelbaum \cite{mandelbaum:1972aa,mandelbaum:1973aa} and Simpson \cite{simpson:2010aa} (see also \cite{biswas:2023aa}). As usual, the idea of branching involves considering geometric structures on $M\backslash S$, where $S$ is a discrete set of points on $M$. Furthermore, one imposes regularity conditions around these points, so that the structures obtained are still controlled by the geometry of $M$.
	
	A \emph{branched projective structure} is a maximal atlas $\mathcal{A}=\{(U_\alpha,\phi_\alpha:U_\alpha\to\cp)\}$ whose charts are topological singly-branched $r$-coverings from topological disks $U_\alpha$ of $M$ onto open sets of $\cp$, and whose transition functions are Möbius transformations. This is to say that the charts $\phi_\alpha$ are $r$-coverings, $r\geq1$ an integer, except possibly at a single point of their domain $p_\alpha$. Further we require that $p_\alpha$ is the only point on its $\phi_\alpha$-pre-image, i.e., $\phi_\alpha^{-1}(\phi_\alpha(p_\alpha))=\{p_\alpha\}$.  The set of points $p_\alpha$ is well defined and discrete. For compact $M$ it is thus finite and we denote it by $S=\{p_k\}_{k=1}^{n}$; the points $p_k$ are called the \emph{branching points} of $\mathcal{A}$.
	
	Under these conditions, the atlas $\mathcal{A}$ will induce a Riemann surface structure $X$. We say that the branched projective structure $\mathcal{A}$ is \emph{compatible} with $X$. Using the complex coordinates of $X$, the condition on the charts is equivalent to $\phi_\alpha$ being written as the map $z\mapsto z^r $, with $r>1$ only at the branching points $p_k$.
	The integer $r-1$ is independent of coordinates and thus well defined for $\mathcal{A}$. It is called the \emph{order of $\mathcal{A}$ at $p_k$}, and denoted by $\ord_\mathcal{A}(p_k)$. Each branched projective structure compatible with $X$ comes together with the so-called \emph{branching divisor} $D$. This is the effective divisor defined by 
	$$
	D=\sum_{k=1}^{n} \ord_\mathcal{A}(p_k) \cdot p_k.
	$$
	It is trivial if and only if the structure is a projective structure as defined in the previous section. To be clear we will call such structures unbranched.
	
	Equivalently a branched projective structure can be given by
	a pair $(\rho,d)$ where $\rho:\pi_1(M)\to \psl(2,\CC)$ is the holonomy of the structure, and $d:\widetilde{M}\to \cp$ is the $\rho$-equivariant developing map, which is now allowed to have vanishing derivative of order $r-1$ at the lifts of $p_k$ to $\widetilde{M}$. The pair $(\rho,d)$ is again uniquely defined up to the action of $\psl(2,\CC)$ by conjugation on $\rho$ and by post-composition on $d$. 
	
	Building again the flat projective bundle $\PP E_\rho\to M$ the developing map corresponds to a section $s$ which is generically transverse to the flat connection, in the sense that the section is non-horizontal except at the points $p_k$, where the flat connection annihilates $s$ to order $r-1$. The equivalence here is again given by the usual action of the gauge group on flat connections and sections.
	
	We then consider any of these spaces of structures modulo the action of diffeomorphisms locally in $\psl(2, \CC)$ and isotopic to the identity, 
	and denote them by
	$\mathscr{B}(M)$. We also denote the space of branched projective structures compatible with $X$ by $\mathscr{B}(X)$. 
	
	As in the unbranched case we have a map 
	\begin{equation*}
		\mathscr{B}(M) \to \mathscr{T}(M)
	\end{equation*}
	sending a branched projective structure to its induced Riemann surface structure $X$, and whose fiber over $X$ is $\mathscr{P}(X)$.

	\subsection{Gunning's criterion} \label{section Gunning}
	The relation between opers and projective structures is well known. Here we recall the relation for the branched case.
	We let $(\rho,d)\in \mathscr{B}(X)$ be a branched projective structure compatible with the Riemann surface $X$, and with branching divisor $D$. We will look only at those structures whose representation lifts to $\sl$, that is, we assume there is a representation $\widetilde{\rho}:\pi_1(M)\to \sl$ which covers $\rho$ with respect to the quotient map $\Proj:\sl \to \psl(2,\CC)$. This happens if and only if $\deg(D)$ is even (\cite[Theorem 2]{mandelbaum:1973aa} and also \cite[Corollary 11.2.3.]{gallo:2000aa}), and the lift is non-unique in general. In particular, for unbranched projective structures $\rho$ always lifts.
	
	In this situation we can repeat the construction of the bundle but using $\widetilde{\rho}$ and $\CC^2$ as a fiber, thus obtaining a flat $\sl$-bundle $E\to M$, with flat linear connection $\nabla$, and holomorphic structure $\DDD_E:=\nabla^{0,1}$ 
	as a bundle on the underlying Riemann surface $X$. It comes
	with a holomorphic line subbundle $L$, corresponding to the section $s$ or, equivalently, to the branched developing map $d$. The branched transversality condition is precisely that $L$ be a non-horizontal subbundle, i.e., $TL \nsubseteq HE$  with $HE\subset TE$ the horizontal distribution determined by the flat connection on $E$, except at the branching points $p_k$ where the order of contact is giving by the branching order.

	Another way of expressing the transversality condition comes from looking
	at the
	composition $q\circ \nabla: \mathcal{O}{(L)} \to \mathcal{O}{(E/L\otimes K)}$,
	where $q:E\otimes K \to E/L \otimes K$ is the quotient map.
	This is $\mathcal{O}$-linear and thus defines a
	holomorphic map of line bundles $\beta_L:L\to E/L\otimes K$,
	called \emph{the second fundamental form}
	of $L$ in $E$. The subbundle is
	non-horizontal precisely if $\beta_L$ is non-zero, and
	the zeros of this map are the points $p_k$ where the structure
	branches. Note that, since $\det E = \mathcal{O}$ and $\rank(E)=2$,
	we have $E/L\cong L^{-1}$.
	Thus $\Hom(L,E/L\otimes K)\cong\Hom(L, L^{-1}  K)\cong L^{-2}  K$, and so we can see $\beta_L$ as a section of a line bundle, $\beta_L\in \holosecs{L^{-2}K}$. The branching divisor is precisely $\Div(\beta_L)$.
	
	In conclusion, every branched projective structure compatible with $X$, and whose representation lifts to $\sl$, gives rise to a flat $\sl$-vector bundle together with a non-horizontal holomorphic line subbundle $L$, i.e., one with non-zero second fundamental form $\beta_L$. Conversely, any such data will yield a branched projective structure.
	These objects are again only determined up to gauge equivalence, but now the ambiguity introduced by lifting the representation means that it is possible for different subbundles $L$ and $L'$ to determine the same projective structure. This happens if and only if $L'=L\otimes S $, for some flat line bundle $S$, since they projectivize to the same section of $\PP E_\rho\to X$.
	We collect these results in the following theorem due, in the unbranched case, to Gunning \cite[Theorem 2]{gunning:1967aa}.
	\begin{theorem}\label{oper}
		Let $E\to X$ be an $\sl$-vector bundle on a Riemann surface $X$ with holomorphic structure $\nabla^{(0,1)}$ given by a flat connection $\nabla$. Then any holomorphic line subbundle $L\subset E$ with non-zero holomorphic second fundamental form $\beta_L$, determines a branched projective structure compatible with $X$, with branching divisor $\Div(\beta_L)$. Moreover, any other $\sl$-bundle $E'$ with flat connection $\nabla'$ and subbundle $L'$ determines an equivalent projective structure if and only if there is an isomorphism $\Psi:E\otimes S \to E'$ such that $L'=\Psi(L\otimes S)$  and $\nabla\otimes \delta =\Psi^* (\nabla')$ for some holomorphic line bundle $S$ with flat connection $\delta$ of order 2.
	\end{theorem}
	
	\begin{remark}
		Note that, if $\Proj:\sl\to\psl(2,\CC)$ is the quotient map, the holonomy of the structure determined by $(E\to X,\nabla, L)$ is $\Proj\circ \hol(\nabla)$.
		Note also that, given the fact that the connections induced on $\det(E)$ and $\det(E')$ are trivial and $\det(E'){\cong}\det(E) \otimes \delta^{2}$ via $\Phi$ we conclude that the flat connection $\delta$ squares to the trivial connection on $S$. Such flat connections correspond to representations $\pi(M)\to\CC^*$ whose elements have order 2. Since there are only two such elements in the abelian group $\CC^*$, we conclude that there exist $2^{2g}$ such possible connections.
	\end{remark}
	
	This shows that a concept of branched projective structure compatible with $X$ is almost equivalent to that of a partial $\sl$-oper which we introduce next, following \cite{simpson:2010aa}.
	
	\subsection{Partial $\sl$-opers} \label{section partial oper}
	Classical opers were introduced in \cite{drinfeld:1985aa} and the concept was reformulated in modern language in \cite{beilinson:aa}. They are defined as a flat bundle together with a full filtration by holomorphic subbundles, whose induced map on quotients satisfies a transversality condition. In \cite{simpson:2010aa} the filtration is not necessarily full, and the transversality condition is replaced by the $gr$-semistability of the associated graded Higgs bundle, giving rise to the notion of \emph{partial oper}.
	\begin{definition} \cite{simpson:2010aa}
		Let $X$ be a Riemann surface and consider an $\sl$-bundle $E$ with flat connection $\nabla$ together with a filtration $$0\subset L \subset E$$ by a $\nabla^{(0,1)}$-holomorphic subbundle $L$. Let $\beta_L:L\to E/L\otimes K$ be the $\mathcal{O}$-linear map induced by $\nabla$.
		The \emph{associated graded} is the Higgs bundle $(\Gr(E),\theta)$ where $$ \Gr(E)=L\oplus E/L \cong L\oplus L^{-1} \hspace{5pt} \text{ and } \hspace{5pt} \theta= \begin{pmatrix}
			0 & 0\\
			\beta_L & 0
		\end{pmatrix}.$$
		If this Higgs bundle is semistable we call the filtration \emph{a partial $\sl$-oper}.
		Two partial opers are equivalent if their flat bundles are isomorphic by a gauge transformation that preserves the filtration.
	\end{definition}
	\begin{remark} \label{oper remark}
		Note that if $\deg(L)>0$ then $L$ is the maximal destabilizing subbundle of $\Gr(E)$. This means that $(\Gr(E),\theta)$ is (semi)stable as a Higgs bundle precisely if $L$ is not preserved by $\theta$, and this happens if and only if $\beta_L \not\equiv 0$. In this case, the definitions of branched projective structure compatible with $X$ and of partial oper are the same, by Theorem~\ref{oper}, and they just require $\beta_L$ to be non-zero.
		In the case $\deg(L)=0$ though, the bundle $\Gr(E)$ is semistable (as a holomorphic bundle).  Thus $(\Gr(E),\theta)$ is semistable as a Higgs bundle even if $\beta_L$ is zero. In this situation, the definition of partial oper includes more objects than the branched projective structures. These structures will correspond to the partial opers with non-zero $\beta_L$.
	\end{remark}
	\begin{remark}
		{We remark that Simpson's definition in \cite{simpson:2010aa} allows for filtrations which are not full. In particular, $0\subset E$, with $E$ semistable as a holomorphic bundle, is considered a partial oper structure in that paper, but not here. }
	\end{remark}
	\begin{remark}
		The case of classical, 
		or \emph{full}, 
		opers is obtained when $\beta_L:L\to L^{-1}K$ is an isomorphism. In this case $L^2\cong K$ and $\deg(L)=g-1$, and $(\Gr(E),\theta)$ is automatically stable, since $\beta_L\neq0$. These correspond to unbranched projective structures compatible with $X$.
	\end{remark}
	For now, we will be interested in the case of
	$\deg(L)>0$. Then, for a Riemann surface $X$, a
	branched projective structure compatible with $X$ is the same as a
	partial $\sl$-oper, with the caveat that equivalence of projective
	structures is slightly weaker than that of opers. While for opers the filtration must be preserved by gauge equivalence, in the
	situation of projective structures the gauge equivalence is allowed to
	twist the subbundle. In any case, regardless of degree, a partial
	$\sl$-oper on $X$, with non-zero $\beta_L$, determines a
	branched projective structure compatible with $X$.
	
	We also recall that the structure of a partial $\sl$-oper on $E$, in particular, realizes $E$ as an extension of the $\nabla^{(0,1)}$-holomorphic line bundles, namely of $L$ by $E/L\cong L^{-1}:$ 
	$$ 0 \to L \to E{\to} L^{-1} \to 0.$$
	These extensions are classified by an element of the cohomology group $\mathrm{H}^1(\Hom(L^{-1},L))\cong \mathrm{H}^1(L^2)$, and thus represented in Dolbeault cohomology by a $(0,1)$-form with values in $L^2$. In any $\cifty$ decomposition of $E$ of the form $E=L\oplus L^{-1}$, the holomorphic structure $\nabla^{(0,1)}$ can be written as $$\nabla^{(0,1)} = \begin{pmatrix}
		\DD_L & \omega\\
		0 & \DD_{L^{-1}}
	\end{pmatrix},$$ since $L$ is holomorphic, and the class $[\omega]\in \mathrm{H}^1(L^2)$ is called is the \emph{extension class} of the partial $\sl$-oper $(E,\nabla)$.
	
	\subsection{The conformal limit is a partial oper} \label{section conformal limit is oper}
	
	We now observe that the conformal limit calculated in Theorem~\ref{conflimit} defines a partial oper. 
	This was already proved in \cite{collier:2018aa} in greater generality. We include the proof for two reasons: firstly, because in the $\slr$-case we can give a particularly transparent proof using our explicit argument for the existence of the conformal limit and, secondly, because this sets the stage for the construction of branched projective structures in Section~\ref{section branched projective structures coming from the limit}.
	
	\begin{theorem}
		The $\hbar$-conformal limit $\nabla_{\hbar,0}$ of the $\slr$-Higgs bundle $E=L\oplus L^{-1}$, with $1\leq\deg(L)\leq g-1$,  and Higgs field $\Phi=\begin{psmallmatrix}
			0 & \alpha\\
			\beta &0
		\end{psmallmatrix}$, where $\alpha \in \holosecs{L^2K}$ and $0\neq \beta \in \holosecs{L^2K}$ is a partial $\sl$-oper
		\begin{equation*}
			0 \subset L \subset (E,\nabla_{\hbar,0})
		\end{equation*}
		with second fundamental form $\hbar^{-1} \beta$.
	\end{theorem}
	\begin{proof}
		The expression \eqref{eq:conflimit} shows that $\nabla_{\hbar,0}^{(0,1)}$ preserves $L$ which is therefore a holomorphic subbundle as required. Moreover, the second fundamental form is the lower left-hand corner of the matrix in \eqref{eq:conflimit} which is indeed the non-zero holomorphic section $\hbar^{-1} \beta$.
	\end{proof}
	\begin{remark}
		This is the same as saying that the conformal limit yields a branched projective structure compatible with $X$. Its branching divisor is precisely the divisor of $\beta$ in the Higgs field. 
	\end{remark}
	\begin{remark} \label{alpha zero}
		Observe also that in the case where $\alpha=0$, the Higgs field $\Phi$ lies in the nilpotent cone. In this situation, the entire family $\nabla_{\hbar, R}$ in (\ref{R family}) has an explicit form similar to $\nabla_{\hbar,0}$. So the proof of the theorem actually yields that $L$ is a holomorphic subbundle with non-zero second fundamental form. Thus every $\nabla_{\hbar, R}$, for $\alpha=0$, is a partial oper. We will see that there are more general conditions under which $\nabla_{\hbar, R}$ is a branched projective structure.
	\end{remark}
	\begin{remark}
		The extension class of the limit partial oper is represented by the upper right-hand corner of $\nabla_{\hbar, 0}^{0,1}$ which reads $[\hbar \conj{\beta} h_0^{-2}]\in \mathrm{H}^{0,1}(L^{2}) \cong \mathrm{H}^1(L^2)$.
	\end{remark}
	
	\begin{remark}
		\label{example hitchin component}
		Since the branching divisor of the projective structure given by the conformal limit is the divisor of zeros of $\beta$, we see that the $\hbar$-conformal limit $\nabla_{\hbar,0}$ of the $\sl$-Higgs bundle $(\DDD_E, \Phi)$ is a full oper if and only if $\beta$ is nowhere vanishing, i.e., if and only if $(\DDD_E, \Phi)$ lies in a Hitchin component.
	\end{remark}

	It is important to note that partial opers correspond only to branched projective structures which are compatible with a fixed Riemann surface structure $X$. To study what happens along the conformal limit, i.e., as $R\to 0$ in $\nabla_{\hbar,R}$, and to check that these connections actually correspond to branched projective structures we need to vary the structure $X$. To this purpose, we digress slightly on the theory of the Beltrami equation and Teichmüller space.

	\subsection{Beltrami differentials and complex structures} \label{section beltrami differentials}
	There are several approaches to the description of the Teichmüller space $\mathscr{T}(M)$ of a closed surface $M$. This is the space of marked complex structures on $M$ and it is classically identified with the unit $L^1$-ball in the space of holomorphic quadratic differentials $\holosecs{K^2}$ for some fixed Riemann surface structure via Teichmüller's embedding, see for example \cite[Theorem 2.9]{daskalopoulos:2007wy}. Its construction relies on the use of Beltrami differentials to change the complex structure. We follow the notation in \cite[IV.1.4]{lehto:1987aa}.
	
	\begin{definition}
		Let $X$ be Riemann surface with canonical bundle $K$. A \emph{Beltrami differential} $\mu$ is a (smooth) section of $\overline{K}\otimes K^{-1}$ whose $\sup$-norm is strictly bounded by 1, i.e., an element of the set
		\begin{equation}
			\beltrami (X)=\left\{ \mu \in \secs{\overline{K}\otimes K^{-1}}=\forms{(-1,1)}{X}\,\vert\, ||\mu||_\infty<1\right\}.
		\end{equation}
	\end{definition}
	
	\begin{remark}
		Note that the transformation law for the coordinates of $\mu$ is $\mu'(z')=\mu(z) \frac{dz'/dz}{d\conj{z}'/d\conj{z}}$. As such, the value of $|\mu(z)|$ is independent of the holomorphic coordinate chart and it is a well-defined quantity whose supremum we denote by  $||\mu||_\infty$.
	\end{remark}
	
	A Beltrami differential can be used to build a complex atlas for a different complex structure in the following way. Let $z$ be a coordinate for $X$ and let $\mu=\mu(z) d\conj{z} \otimes \frac{\DD}{\DD z}\in\secs{\overline{K}\otimes K^{-1}}$. Consider the local \emph{Beltrami equation} on a contractible open set $U\subset X$ for a function $v:U\to \CC$
	\begin{equation}
		\DD_{\conj{z}}v=\mu(z) \DD_z v.
	\end{equation}
	In this situation, $\mu(z)$ is called the \emph{parameter} of the equation.
	The classical existence result states that such a $v$ exists if $||\mu||_\infty<1$ on $U$. Furthermore, such function $v$ is a diffeomorphism onto some open set of $\CC$. Now, we can check the condition for existence on each open set $U$ of $X$, and this happens precisely if $\mu$ is a Beltrami differential. If we collect all such local functions $v$ together we can check that they form a complex atlas for a new complex structure denoted by $X_\mu$. A contemporary description of this result can be found for example in \cite[Theorem~4.8.12]{hubbard:2006ve}.
	
	\begin{theorem} \label{Xmu}
		Let $X$ be a Riemann Surface. Then any Beltrami differential $\mu \in \beltrami(X)$ determines a complex structure $X_\mu$ whose local charts are the solutions of the Beltrami equation with parameter $\mu(z)$. 
	\end{theorem}
	
	In conclusion, if $X$ has coordinate $z$, then $X_\mu$ has coordinate $v$ such that $\DD_{\conj{z}}v = \mu(z) \DD_z v$. Writing $\nu=\DD_z v$ we can observe that the bases of $(1,0)$ and $(0,1)$-forms of $X_\mu$ are given with respect to $X$ by
	\begin{align} 
		\label{coordinates new} dv&= \nu(dz + \mu (z)d\conj{z})\\
		\label{coordinates new2} d\conj{v}&= \conj{\nu} (d\conj{z} + \conj{\mu}(z) dz),
	\end{align}
	with $\nu\neq0$ since $v$ is a local diffeomorphism.
	
	\begin{remark}
		It is actually true that any complex structure up to biholomorphism isotopic to the identity arises in this way. (This can be deduced, for example, from the surjectivity statement of \cite[Theorem 2.9]{daskalopoulos:2007wy}, together with the fact that the map there factors through a map out of $\beltrami(X)$). In particular, there is an identification of $\mathscr{T}(M)$ with the space of Beltrami differentials $\beltrami(X)$ modulo the equivalence relation where $\mu\sim\mu'$ if there is a biholomorphism $X_\mu \to X_{\mu'}$ isotopic to the identity.
	\end{remark}
	
	\section{Branched projective structures coming from the conformal limit}
	\label{section branched projective structures coming from the limit}
	We are now ready to carry out the main construction of branched projective structures coming from the conformal limit.
	We will define an appropriate Riemann surface structure $X_\mu$ for which $\nabla_{\hbar,R}$ is a partial oper, i.e., it defines a branched $\cp$-structure compatible with $X_\mu$. This construction will carry through provided the Higgs field
	$\Phi=\begin{psmallmatrix}
		0 & \alpha\\
		\beta &0
	\end{psmallmatrix}$ lies in the special locus where the divisors of zeros satisfy $\Div(\alpha)\geq \Div(\beta)$ and given also that $|\hbar^2 R^2|\leq1$. In this section, we shall assume that $\deg(L)> 0$. The case $\deg(L)=0$ will be treated in Section~\ref{degree zero section}.

	\begin{theorem} \label{main theorem}
		Fix $R\in\RR^+_0$ and $\hbar \in \CC^*$ such that $|\hbar^2 R^2|\leq1$. Let $L$ be a line bundle of degree $1\leq\deg(L)\leq g-1$ and $h_R$ be a solution of the $R$-scaled vortex equation for $(\alpha,\beta)$, where $\alpha \in \holosecs{L^2K}$ and $0\neq \beta \in \holosecs{L^2K}$. Assume further that $\Div(\alpha)\geq \Div(\beta)$ and let $\mu=\hbar^2 R^2 \frac{\conj{\alpha}}{\beta} h_R^2$. Then $\mu$ is a Beltrami differential. Further  $\nabla_{\hbar, R}$ is a partial oper for $X_\mu$, with branching divisor $\Div(\beta)$. In particular, it determines a branched projective structure compatible with $X_\mu$.
	\end{theorem}
	\begin{remark}
		The case $\Div(\alpha)=\Div(\beta)$ implies that $\deg(L^2K)=\deg(L^{-2}K)$ and thus $\deg(L)=0$, which is excluded by the hypothesis on $L$.
		Moreover, if $\alpha=0$, we consider that the condition $\Div(\alpha)\geq \Div(\beta)$ holds and then $\mu=0$. In this case, everything is compatible with the base Riemann surface structure $X$ since $\nabla_{\hbar,R}$ is a partial oper (cf.\ Remark~\ref{alpha zero}).
	\end{remark}
	
	\begin{proof}
		For $\alpha=0$ there is nothing to prove in view of the preceding remark.  So we treat the case of non-zero $\alpha$.
		Note that $\mu=\hbar^2 R^2 \frac{\conj{\alpha}}{\beta} h_R^2$ is a smooth section of
		\begin{displaymath}
			\conj{L^2 K}\otimes(L^{-2}K)^{-1}\otimes(L^{-2}\conj{L}^{-2})
			\cong \conj{K}\otimes K^{-1}.
		\end{displaymath}
		Thus only $||\mu||_\infty$<1 needs to be checked to prove that $\mu$ is a Beltrami differential. This will be done in Lemma~\ref{lem:sup-mu-proof} below.  Now we can define a complex structure $X_\mu$ by Theorem~\ref{Xmu} whose coordinates are given by solutions $v$ of the Beltrami equation for $\mu$. Using equation (\ref{coordinates new}) one knows that $dv= \nu (dz + \mu(z) d\conj{z})$, with $\nu=\DD_zv\neq 0$. Thus, writing $\alpha=\alpha(z) dz$ and $\beta = \beta(z) dz$, the flat connection $\nabla_{\hbar, R}$ can be written as in (\ref{R family}),
		\begin{equation} \label{mu form of connection}
			\nabla_{\hbar, R}-d=
			\begin{pmatrix}
				* & *\\
				\hbar^{-1} \beta(z) \left(dz + \hbar^2 R^2 \frac{\conj{\alpha}(z)}{\beta(z)} h_R^2(z) d\conj{z}\right) & *
			\end{pmatrix}
			= \begin{pmatrix}
				* & *\\
				\hbar^{-1} \frac{\beta(z)}{\nu} dv & *
			\end{pmatrix}.
		\end{equation}
		This shows in particular that the holomorphic structure
		$\nabla_{\hbar, R}^{(0,1)_\mu}$ on $E$ (as a bundle on $X_\mu$) preserves $L$, since
		\begin{equation*}
			\nabla_{\hbar, R}^{(0,1)_\mu}-\DDD^{\mu}= 	\begin{pmatrix}
				* & *\\
				\hbar^{-1} \frac{\beta(z)}{\nu} dv & *
			\end{pmatrix}^{(0,1)_\mu}=\begin{pmatrix}
				* & *\\
				0 & *
			\end{pmatrix}.
		\end{equation*}
		where $\DDD^{\mu}$ is the $\DDD$ operator of $X_{\mu}$.
		
		This means that $L_\mu\subset E_\mu$ is in fact a
		holomorphic subbundle, where the subscript $\mu$ indicates that we
		are considering holomorphic bundles on $X_\mu$ with the holomorphic
		structure 
		induced by the flat connection $\nabla_{\hbar, R}$.
		
		It remains to show that the $X_\mu$-holomorphic second fundamental
		form $\beta^{\mu}_L$ is non-zero. To this effect, we note that
		$\beta_L^\mu$ is the $\mathcal{O}_{X_\mu}$-localization of
		$q\circ \nabla_{\hbar, R} :\mathcal{O}_{X_\mu}(L_\mu)\to
		\mathcal{O}_{X_\mu} (E_\mu/L_\mu \otimes K_\mu)$, where $K_\mu$ is the
		canonical bundle of $X_\mu$.
		But,
		observing the form of $\nabla_{\hbar,R}$ in (\ref{mu form of
			connection}), this is simply locally given by multiplication by
		$\hbar^{-1}\frac{\beta(z)}{\nu}$, which is non-zero. Further, the
		order of vanishing at each point is precisely the one of $\beta$, thus
		implying that the branching divisor is $\Div(\beta)$.
	\end{proof}
	
	\begin{lemma}
		\label{lem:sup-mu-proof}
		Let $|\hbar^2 R^2|\leq1$ then $|\mu(z)|^2=\left|\hbar^2 R^2 \frac{\conj{\alpha}(z)}{\beta(z)} h_R^2(z) \right|^2<1$ everywhere on $X$.
	\end{lemma}
	
	\begin{proof} \allowdisplaybreaks
		We observe that, as $\Div(\alpha)\geq\Div(\beta)$, $\mu$ is smooth. We consider the function $$u(z)=\log \frac{|\mu|^2}{|\hbar^2 R^2|^2}=\log \left|\frac{\alpha(z)}{\beta(z)}h_R^2(z) \right|^2. $$
		This is simply the logarithm of the norm squared of the section $\frac{\alpha}{\beta} \in \holosecs{L^4}$, where $L$ is given the metric $h_R$ which is a solution of the vortex equations (\ref{vortex eqs}). 
		We will show that $u<0$ everywhere on $M$, thus implying $|\mu(z)|^2<|\hbar^2 R^2|^2$. As, by hypothesis $|\hbar^2 R^2|\leq 1$, the conclusion that $|\mu(z)|^2<1$ everywhere follows. To achieve the inequality $u<0$ we use the maximum principle for elliptic operators (as in Hitchin \cite[proof of Theorem (11.2)]{hitchin:1987vg}, following Li \cite[Claim 6.1]{li:2019aa}), which we set up as follows.
		
		Consider the set 
		$\{z_1, z_2,\cdots, z_k\}$ of zeros of $\frac{\alpha}{\beta}$, which
		is non-empty, because $\Div(\alpha)$ is not everywhere equal to
		$\Div(\beta)$. In this set, the function $u$ has negative singularities, i.e., points where $\lim_{z\to z_j} u(z)=-\infty$. This means we can consider small closed disks around these points where $u(z)$ is as negative as we want. In particular, we can get disks where $u<0$ (strict inequality). When we remove these disks from the surface $M$, we get a manifold with boundary $U$. The function $u$ is smooth on the interior of $U$ and continuous up to the boundary $\DD U$, where $u<0$. These are part of the conditions to apply the maximum principle. Note that we only need to show that $u<0$ on the interior of $U$, since in $M\backslash U$ the inequality already holds (possibly with $u=-\infty$).
		
		Suppressing from the notation the dependence on $z$, we see that $e^u=h_R^4|\frac{\alpha}{\beta}|^2$, and thus $e^{u/2}=h_R^2|\frac{\alpha}{\beta}|.$ 
		So, writing $u= \log \frac{\alpha \conj{\alpha}}{\beta \conj{\beta}} h_R^4$, we can calculate that, away from the zeros of $\alpha$ (in particular in $U$),
		\begin{align*} 
			\DDDz \DDz u &= \DDDz \DDz \log \frac{\alpha}{\beta} +\DDDz \DDz \log \frac{\conj{\alpha}}{\conj{\beta}}+ 4 \DDDz \DDz \log h_R\\
			&= 0 + 0 + 4 \DDDz \DDz \log h_R \qquad \qquad  \text{ (because } \tfrac{\alpha}{\beta} \text{ is holomorphic)}\\
			&=4 R^{2}\left(|\alpha|^2h_R^2 - |\beta|^2 h_R^{-2}\right) \qquad \qquad  \text{ (by the vortex equation (\ref{vortex eqs}}))\\
			&= 4 R^{2}|\alpha \beta| \, \left(\frac{|\alpha|}{|\beta|}h_R^2 - \frac{|\beta|}{|\alpha|} h_R^{-2}\right)\\
			&= 4 R^{2}|\alpha \beta| \left(e^{u/2} - e^{-u/2}\right) = 8 R^{2}|\alpha \beta| \sinh(u/2).
		\end{align*}
		Recalling that the Laplacian of $X$ is $\Delta=\frac{4}{g_0} \DDDz \DDz $, where $g_0$ is the metric of constant negative curvature $-4$ (cf.\ Example~\ref{example hitchin}), we have equivalently \begin{equation*}
			\Delta u = 32 R^2 \frac{|\alpha \beta |}{g_0} \sinh(u/2).
		\end{equation*}
		This is a $\sinh$-Gordon type equation, which can be written as
		\begin{equation*}
			L[u]= \Delta u - 32 R^2 \frac{|\alpha \beta |}{g_0} \frac{\sinh(u/2)}{u} \, u =0,
		\end{equation*}
		since $\lim_{u\to 0} \frac{\sinh(u/2)}{u}=\frac{1}{2} $ is finite. Here $L=\Delta - c$ is a linear differential operator, where $c=32 R^2 \frac{|\alpha \beta |}{g_0} \frac{\sinh(u/2)}{u}$. In particular, $c\geq0$, since $\sinh(u/2)$ and $u$ have the same sign. This implies that $u$ is a solution of a linear partial differential equation which is uniformly elliptic on $U$, and has $c\geq0$. Since $u\leq 0$ on the boundary $\DD U$, we are thus in the conditions of the classical maximum principle \cite[Theorem 3.5]{gilbarg:2001aa} which then implies that either $u$ is constant or it cannot attain a non-negative maximum in the interior of $U$. As $u$ cannot be constant (since $\alpha$ has zeros but it is non-zero) we conclude that $u<0$ in the interior of $U$, which finishes the proof.
	\end{proof}
	We can also calculate the extension class of this partial $\sl$-oper $\nabla_{\hbar, R}$.
	
	\begin{proposition}
		The extension class of the partial $\sl$-oper $(E,\nabla_{\hbar,R})$ over $X_\mu$, with $\mu=\hbar^2 R^2 \frac{\conj{\alpha}}{\beta} h_R^2$, is the class in $\mathrm{H}^1(X_\mu,L_\mu^2)$ represented in Dolbeault cohomology by\begin{equation*}
			\omega = \left( \frac{1-|\mu|^2/|\hbar^2 R^2|^2}{1-|\mu|^2} \right)\hbar R^2 \conj{\beta}(z) h_R^{-2} \frac{d \conj{v}}{\conj{\nu}}.
		\end{equation*}
	\end{proposition}
	\begin{proof}\allowdisplaybreaks
		We need to calculate the $(0,1)_\mu$-part of the upper right entry of $\nabla_{\hbar, R}$. Comparing with (\ref{R family}), this is $\omega= (\hbar^{-1}\alpha(z) dz + \hbar R^2 \conj{\beta}(z) h_R^{-2} d\conj{z})^{(0,1)_\mu}$. To calculate, we note that we can invert equations (\ref{coordinates new}) and (\ref{coordinates new2}) to get\vspace{-10pt}
		\begin{align*}
			dz&= \frac{1}{1-|\mu|^2}\left(\frac{d v}{\nu} - \mu \frac{d \conj{v}}{\conj{\nu}}  \right)\\
			d\conj{z}&=\frac{1}{1-|\mu|^2}\left( -\conj{\mu} \frac{d v}{\nu}  +   \frac{d \conj{v}}{\conj{\nu}}\right).
		\end{align*}
		This means $\omega$ has a term coming from $\alpha$ which is $\hbar^{-1}\alpha(z) \frac{- \mu}{1-|\mu|^2} \frac{d \conj{v}}{\conj{\nu}} $ and another one coming from $\beta$ which is $\hbar R^2 \conj{\beta}(z) h_R^{-2} \frac{1}{1-|\mu|^2} \frac{d \conj{v}}{\conj{\nu}}$. This implies \vspace{-9pt}
		\begin{align*}
			\omega&= \left(\hbar R^2 \conj{\beta}(z) h_R^{-2} -\mu\hbar^{-1}\alpha(z) \right)\frac{1}{1-|\mu|^2} \frac{d \conj{v}}{\conj{\nu}}\\
			&= \hbar R^2 \conj{\beta}(z) h_R^{-2}  \left( 1-\mu\frac{\hbar^{-1}\alpha(z)}{\hbar R^2 \conj{\beta}(z) h_R^{-2}} \right)\frac{1}{1-|\mu|^2} \frac{d \conj{v}}{\conj{\nu}}\\
			&=\hbar R^2 \conj{\beta}(z) h_R^{-2}  \left( 1-\mu\frac{\conj{\hbar}^2 R^2}{\conj{\hbar}^2 R^2}\frac{\alpha(z) h_R^{2}}{\hbar^2 R^2 \conj{\beta}(z)} \right)\frac{1}{1-|\mu|^2} \frac{d \conj{v}}{\conj{\nu}}\\
			&=\hbar R^2 \conj{\beta}(z) h_R^{-2}  \left( 1-\mu\frac{\conj{\mu}}{|\hbar^2 R^2|^2} \right)\frac{1}{1-|\mu|^2} \frac{d \conj{v}}{\conj{\nu}}\\
			&=\left( \frac{1-|\mu|^2/|\hbar^2 R^2|^2}{1-|\mu|^2} \right)\hbar R^2 \conj{\beta}(z) h_R^{-2} \frac{d \conj{v}}{\conj{\nu}}.
		\end{align*}
	\end{proof}
	
	\section{The case of zero degree}
	\label{degree zero section}
	In this section we consider the case $\deg(L)=0$, where
	the construction has slightly different features. In particular, as already noted
	in Remark~\ref{degree zero}, a different argument is required for the
	existence of the conformal limit and we start with this. 
	
	\subsection{The conformal limit} \label{section the conformal limit degree zero}
	
	The conformal limit requires the polystability of the $\slr$-Higgs bundle. So we begin by studying the stability in this case.
	
	\begin{proposition}
		\label{stability degree zero}
		Let $L$ be a line bundle with $\deg(L)=0$. Consider the $\slr$-Higgs
		bundle $(E=L\oplus L^{-1}, \Phi=\begin{psmallmatrix}
			0 & \alpha \\
			\beta & 0
		\end{psmallmatrix})$, with $\alpha\in\holosecs{L^2K}$ and
		$\beta \in \holosecs{L^{-2}K}$. Then
		\begin{enumerate}[label=\roman*)]
			\item if both $\alpha=0$ and $\beta=0,$ 
			$(E,\Phi=0)$ is strictly polystable as an $\sl$-Higgs bundle;
			\item if both $\alpha \neq0$ and $\beta \neq 0$, $(E,\Phi)$ is
			polystable. If further $\alpha$ and $\beta$ are not proportional,
			$(E,\Phi)$ is stable;
			\item if one of $\alpha$ and $\beta$ is zero but not the other one,
			$(E,\Phi)$ is unstable.
		\end{enumerate}
	\end{proposition}
	
	\begin{proof}
		Case $i)$ is immediate.
		For case $iii)$ we note that if only one of
		$\alpha$ or $\beta$ is non-zero, then the Higgs field
		$\Phi$ is nilpotent. This means $\Phi$ is not diagonalizable,
		and so $(E,\Phi)$ is not a direct sum of Higgs line bundles. This
		means it is not strictly polystable. It is also not stable, since
		either $L$ or $L^{-1}$ is $\Phi$-invariant.
		We are left
		with case $ii)$.  Since $E$ is polystable as a bundle, only degree
		zero subbundles can destabilize the Higgs bundle $(E,\Phi)$.
		Suppose there is a $\Phi$-invariant holomorphic line subbundle $S$ of degree
		zero. Write $s_1:S\to L$ and $s_2:S\to L^{-1}$ for the maps
		induced by the inclusion $s:S\hookrightarrow E=L\oplus
		L^{-1}$. Both of these maps are non-zero because neither $L$ nor
		$L^{-1}$ is $\Phi$-invariant. Hence (since
		$\deg(S)=\deg(L)=0$) we have $s_1:S^{-1}L\cong\mathcal{O}$ and $s_2^{-1}:SL\cong\mathcal{O}$. Therefore $s_1/s_2:L^2\cong\mathcal{O}$.
		Now, the subbundle $S$ being $\Phi$-invariant means
		that $\Phi (s) = c s$ for a non-zero section $c$, i.e.,
		\begin{equation*} \Phi (s)
			= \begin{pmatrix}
				0 & \alpha \\
				\beta & 0
			\end{pmatrix}
			\begin{pmatrix}
				s_1 \\s_2
			\end{pmatrix}
			= \begin{pmatrix}
				\alpha s_2 \\
				\beta s_1
			\end{pmatrix} = \begin{pmatrix} c s_1 \\c s_2
			\end{pmatrix}=c s.
		\end{equation*}
		Thus $c\alpha s_2^2=c^2 s_1s_2 = c\beta s_1^2$ and, in view of the isomorphism $s_1/s_2:L^2\cong\mathcal{O}$ we conclude that $\alpha$
		and $\beta$ are proportional sections of $L^2K \cong L^{-2}K\cong K$.
		Finally, we can include $S^{-1}$ in $L\oplus L^{-1}$ using
		$\left(\begin{smallmatrix}
			s_2^{-1}\\
			s_1^{-1}
		\end{smallmatrix}\right)$ and, since
		\begin{displaymath}
			\begin{pmatrix}
				0 & \alpha \\
				\beta & 0
			\end{pmatrix}
			\begin{pmatrix}
				s_2^{-1} \\s_1^{-1}
			\end{pmatrix}
			= \begin{pmatrix}
				\alpha s_1^{-1} \\
				\beta s_2^{-1}
			\end{pmatrix}
			= \begin{pmatrix}
				c s_2^{-1} \\
				c s_1^{-1}
			\end{pmatrix}
		\end{displaymath}
		by the above calculation, we conclude that $S^{-1}$ is a $\Phi$-invariant complement to $S$. 
		
		In conclusion, if there is a destabilizing $\Phi$-invariant subbundle $S$, then the sections $\alpha$ and $\beta$ are proportional and $(E,\Phi)$ decomposes as the direct sum of Higgs bundles $ S\oplus S^{-1}$, with the induced Higgs fields. Thus $(E,\Phi)$ is strictly polystable. Otherwise, there are no such subbundles and $(E,\Phi)$ is stable. 
	\end{proof}
	
	Thus, in the case $\deg(L)=0$, the conformal limit can
	be analyzed using the solution of the scaled vortex equations for
	either both $\alpha=0=\beta$ or both non-zero. In this special case,
	since the bundle $E=L\oplus L^{-1}$ itself is stable,
	the limit of the solution $h_R$ of the scaled vortex equations as
	$R\to 0$ does exist, and it is simply a metric $h_0$ of zero curvature
	on $L$. The conformal limit is then directly calculated by taking the
	limit in the family (\ref{R family}) and it is
	$\nabla_{\hbar,0} = A_{h_0} + \hbar^{-1}\Phi$, where $A_{h_0}$
	is the diagonal Chern connection associated to the Hermitian metric
	$h_0$. It is a partial oper (since $E=\Gr(E)$ and it is semistable),
	and, when $\beta\neq0$, it defines a branched complex projective
	structure compatible with $X$, just as in the previous section.
	In conclusion, we have the following.
	
	\begin{theorem}
		Let $(E=L\oplus L^{-1}, \Phi=\begin{psmallmatrix}
			0 & \alpha \\
			\beta & 0
		\end{psmallmatrix})$
		be a polystable $\slr$-Higgs
		bundle with $\deg(L)=0$.
		Then the $\hbar$-conformal limit of $(E,\Phi)$ exists and it has
		the structure of a partial
		$\sl$-oper. 
	\end{theorem}
	
	\begin{remark}
		Proposition~\ref{change solution} is no longer necessary to change the
		solution. This is consistent with the fact that in
		this case the vortex equations for $(\beta,0)$
		do not have a solution.
	\end{remark}

	\subsection{The projective structures coming from the conformal limit} \label{section projective structures degree zero}
	
	In the zero degree case
	the condition $\Div(\alpha)\geq \Div (\beta)$
	implies that $\alpha = k \beta$, $k\in \CC^*$, since both $\alpha$ and $\beta$ must have the same number of zeros counted with multiplicity. We are thus in the situation $ii)$ of Proposition~\ref{stability degree zero}, where $(E=L\oplus L^{-1}, \Phi=\begin{psmallmatrix}
		0 & k \beta \\
		\beta &0
	\end{psmallmatrix})$ is a polystable Higgs bunlde. Note, in particular, that the condition implies that $L^4\cong \mathcal{O}$ since $k=\frac{\alpha}{\beta}\in \holosecs{L^4}$.
	The proof of Theorem~\ref{main theorem} will carry through, except now the function $u$ will read
	$$u=\log \left|\frac{\alpha(z)}{\beta(z)}h_R^2(z) \right|^2=\log \left|k h_R^2(z) \right|^2$$ and it will be constant by the maximum principle.
	Of course in this case one can directly check that $h_R=|k|^{-1/2}$ is a solution of the $R$ scaled-vortex equations, for if $\alpha = k \beta$, the equation reads
	\begin{equation*}
		\DDDz \DDz \log h_R = R^2\left(|\alpha|^2h_R^2 -|\beta|^2h_R^{-2}\right)= R^2\left(|k|^2 |\beta|^2h_R^2 -|\beta|^2h_R^{-2}\right).
	\end{equation*}
	The right hand side when $h_R=|k|^{-1/2}$ is $R^2\left(|k|^2 |\beta|^2|k|^{-1} -|\beta|^2|k|\right)=0$ which is precisely $\DDDz \DDz \log |k|^{-1/2}=0$.
	The Beltrami differential will now be $ \mu=\hbar^2 R^2 \conj{k} \frac{\conj{\beta}}{\beta} h_R^2=\hbar^2 R^2 \frac{\conj{k}}{|k|} \frac{\conj{\beta}}{\beta}$. The subbundle $L$ will determine a partial oper structure and, when $\beta\neq0$, a complex projective structure compatible with $X_\mu$. 
	Thus we have the following result.
	
	\begin{theorem} \label{main theorem degree zero}
		Fix $R\in\RR^+$ and $\hbar, k \in \CC^*$ such that $|\hbar^2 R^2|<1$. Let $L$ be a line bundle such that $L^4\cong \mathcal{O}$ and consider the polystable $\slr$-Higgs bundle $(L\oplus L^{-1},\Phi=\begin{psmallmatrix}
			0 & k \beta \\
			\beta &0
		\end{psmallmatrix})$, where $0\neq \beta \in \holosecs{L^2K}$.
		Then the family (\ref{R family}) is
		\begin{equation*}
			\nabla_{\hbar,R}=d+\begin{pmatrix}	
				0& \hbar^{-1} k \beta  + \hbar R^2 \conj{\beta} |k| \\
				\hbar^{-1} \beta + \hbar R^2 \frac{\conj{k}}{|k|}\conj{\beta}  & 0
			\end{pmatrix}.
		\end{equation*}
		Define $\mu=\hbar^2 R^2 \frac{\conj{k}}{|k|}
		\frac{\conj{\beta}}{\beta}$. Then $\mu$ is a Beltrami differential.
		Further  $\nabla_{\hbar, R}$ determines a branched projective
		structure compatible with $X_\mu$ with branching divisor 
		$\Div(\beta)$. Its extension class $[\omega]$ is trivial.
	\end{theorem}
	
	\section{Geometric interpretation of results} \label{section geometric
		interpretation of the results}
	
	\subsection{Curves in $\mathscr{B}(M)$} \label{section curves}
	Let $(L\oplus L^{-1}, 
	\begin{psmallmatrix}
		0 & \alpha \\
		\beta & 0
	\end{psmallmatrix})$
	be a polystable $\slr$-Higgs bundle with  $\beta\neq 0$. Assume that $\Div(\alpha)\geq \Div(\beta))$ and $|R^2\hbar^2|\leq1$ (or $|R^2\hbar^2|<1$ if $\deg(L)=0$). Then, by Theorems~\ref{main theorem} and \ref{main theorem degree zero}, the partial oper structure on $E$ given by the filtration $0\subset L \subset E$ and the flat connection $\nabla_{\hbar,R}$ determines a branched projective structure. 
	For each $\hbar$ and $R$, we will denote this structure by $\mathcal{P}^{\alpha}_{\beta}(\hbar,R)\in \mathscr{B}(M)$ or just by $\mathcal{P}(\hbar,R)$ whenever $\alpha$ and $\beta$ are fixed. 
	By construction these branched projective structures lift to $\sl$, so by
	a slight abuse of notation we shall also write $\mathcal{P}(\hbar,R)=(\nabla_{\hbar,R}, L)$, i.e., as a pair consisting of a flat $\sl$-connection and a transverse line subbundle (cf.\ Theorem~\ref{oper}).
	
	Note that the dependence on $\hbar$ and $R$ is continuous, also for $R=0$, since the connection $\nabla_{\hbar,R}$ in the configuration space of flat connections depends continuously on the parameters (Remark~\ref{continuity on configuration space}).
	The structure $\mathcal{P}(\hbar,R)$ is compatible with the Riemann surface $X_\mu$, and thus the forgetful map $\mathcal{T}:\mathscr{B}(M) \to \mathscr{T}(M)$ takes $\mathcal{P}(\hbar,R)\mapsto [X_\mu]$.
	In all cases, regardless of the degree of $L$, the Beltrami differential $\mu$ is given by the expression 
	$$\mu=\hbar^2 R^2 \frac{\conj{\alpha}}{\beta} h_R^2=\hbar^2 R^2 h_R^2 \frac{\conj{\alpha} \conj{\beta}}{|\beta|^2} \frac{|\alpha|}{|\alpha|} =\hbar^2 R^2 h_R^2 \frac{|\alpha|}{|\beta|} \frac{\conj{\alpha} \conj{\beta}}{|\alpha \beta|},
	$$
	where $h_R$ is the solution of the scaled vortex equations for
	$(\alpha,\beta)$, $\beta\neq0$ and it is understood that $\mu=0$ if $\alpha=0$.
	Using Proposition~\ref{change solution} in the case $\deg(L)\neq0$, or the fact that $\alpha=k\beta$ and $h_R=|k|^{-1/2}$, when $\deg(L)=0$ (Proposition~\ref{main theorem degree zero}), for some constant $k\in\CC^*$ we can write $\mu$ as:
	\begin{equation} \label{Beltrami differential all cases}
		\mu=\hbar^2 R^2 h_R^2 \frac{|\alpha|}{|\beta|} \frac{\conj{\alpha} \conj{\beta}}{|\alpha \beta|}
		=
		\begin{cases}
			\hbar^2 R^2  \frac{\conj{k}}{|k|} \frac{\conj{\beta}^2}{|\beta|^2} &\text{if $\deg(L)=0$} \\
			\hbar^2 R^4 h^2 \frac{|\alpha|}{|\beta|} \frac{\conj{\alpha} \conj{\beta}}{|\alpha \beta|}  &\text{if $\deg(L)\neq 0$} 
		\end{cases}
	\end{equation}
	where $h$ is the solution of the unscaled vortex equations for $(R^2\alpha,\beta)$.
	In conclusion, if we fix a pair $(\alpha, \beta)$ with $\Div(\alpha)\geq\Div(\beta)$  and $\beta\neq 0$ and denote the subset of valid parameters $\hbar$ and $R$ by 
	\begin{align*}
		\mathscr{D}=\{(\hbar,R)\in\CC^{*}\times \RR^+_0 \;\vert\; |\hbar^2R^2|\leq1\}, \text{ or }\\
		\mathscr{D}=\{(\hbar,R)\in\CC^{*}\times \RR^+_0 \;\vert\; |\hbar^2R^2|<1\}, \text{ if } \deg(L)=0
	\end{align*}
	
	we get a map $\mathcal{P}^\alpha_\beta:	\mathscr{D} \to \mathscr{B}(M)$ given by $	(\hbar,R) \mapsto \mathcal{P}^\alpha_{\beta}(\hbar,R)$. 
	Thus we have continuous maps
	\begin{align*}
		\mathscr{D} &\to \mathscr{B}(M)\to \mathscr{T}(M)\\
		(\hbar,R) &\mapsto \mathcal{P}^\alpha_{\beta}(\hbar,R) \mapsto X_{\mu(\hbar, R)},
	\end{align*}
	where the associated complex structure $X_{\mu(\hbar, R)}$ is determined by $\mu=\mu(\hbar, R)$ in equation (\ref{Beltrami differential all cases}). We remark that the map is not injective. In particular, $\mu(\hbar,R)=\mu (\hbar', R')$ if $\hbar^2 R^2=	\hbar'^2 R'^2$ for degree zero, or $\hbar^2 R^4=\hbar'^2 R'^4$,  for non-zero degree. 
	We can also fix $\hbar\in\CC^*$, and in that case we obtain a curve $\mathcal{P}^{\alpha}_\beta(R):=\mathcal{P}^{\alpha}_\beta(\hbar, R)$ in $\mathscr{B}(M)$ projecting to a curve $\gamma(R)=X_\mu(\hbar,R)$ in Teichmüller space. In the next section, we study the geometry of this curve which, in the degree zero case, we show to be a Teichmüller geodesic.
	
	\subsection{Teichmüller geodesics and disks} 
	\label{section teichmuller geodesics}
	The Teichmüller space $\mathscr{T}(M)$ has several interesting metrics. One of them is the Teichmüller metric, which is a Finsler metric, and whose distance function is defined using the properties of quasi-conformal mappings. In particular, the distance between $X_\mu,  X_{\mu'} \in \mathscr{T}(M)$ is the smallest possible maximal dilation of a quasi-conformal mapping $f:X_\mu\to X_{\mu'}$ in the isotopy class of the identity of $M$. It is a classical result of Teichmüller that each such class has a unique map that realizes this minimum, the so-called Teichmüller mappings. These maps have complex dilations of the form 
	$$\mu=\frac{\DDD f}{ \DD f}=c \frac{\conj{q}}{|q|}, \qquad 0<c<1, $$ 
	where $c$ is a constant and $q$ a quadratic differential. They allow one to describe the geodesic rays through the origin $X_0$ in $\mathscr{T}(M)$ (cf.\ \cite[Section V.7.7.]{lehto:1987aa}).
	\begin{theorem}
		Let $\mu=t \frac{\conj{q}}{|q|}$ be a Beltrami differential with $q\in \holosecs{X,K^2}$ a quadratic differential, and $t\in[0,1)$. Then $t \mapsto X_{\mu(t)}$ is a geodesic ray in the Teichmüller metric, called the \emph{ray associated with $q$}.
	\end{theorem}
	\begin{remark}
		One can even show that if $t$ is allowed to be in the hyperbolic disk $\mathbb{D}$ the map $t \mapsto X_{\mu(t)}$ is an isometry \cite[Theorem V.9.3.]{lehto:1987aa}. Maps of this form are called complex geodesics or \emph{Teichmüller Disks}.
	\end{remark}
	Observing the constructed $\mu$ in equation (\ref{Beltrami differential all cases}) for the case of $\deg(L)=0$ we immediately conclude that the curve $\gamma(t)$ is a (reparameterization) of a Teichmüller geodesic ray.
	
	\begin{theorem}
		\label{thm:teichmuller-geodesic}
		Let $\mathcal{P}^{k\beta}_\beta(R)$  be the $R$-family of branched projective structures associated to the conformal limit of  $E=L\oplus L^{-1}$, with $L^4\cong \mathcal{O}$,  and Higgs field $\Phi=\begin{psmallmatrix}
			0 & k \beta\\
			\beta &0
		\end{psmallmatrix}$, where $0\neq \beta \in
                \holosecs{L^2K}$ and $k\in \CC^*$. Let $\gamma(R)=X_{\mu(R)}$, with $\mu(R)=\hbar^2 R^2  \frac{\conj{k}}{|k|} \frac{\conj{\beta}^2}{|\beta|^2}$, be the curve of associated Riemann surface structures in $\mathscr{T}(M)$. Assume $|\hbar^2R^2|<1$. Then $\gamma(R)$ is (a reparameterization of) a geodesic ray.
	\end{theorem}
	
	\begin{remark}
		This geodesic ray is associated with the quadratic differential $-\arg(\hbar^2) k\beta^2\in \holosecs{K^2}$.
		By allowing $\hbar$ to vary, we analogously get (a reparametrization of) the Teichmüller disk associated with $k \beta^2$.
	\end{remark}
	\subsection{Reality properties} \label{sections reality properties}
	Let us now study the family $\mathcal{P}^{\alpha}_{\beta}(\hbar, R)$ when $\deg(L)>0$ for some specific parameters. We will show that there exist values of $\hbar$ and $R$ for which the structure $\mathcal{P}^{\alpha}_{\beta}(\hbar, R)$ is a branched hyperbolic structure, a concept which we start by recalling.
	
	A branched projective structure $(\rho, d)$ is called \emph{branched hyperbolic} if its developing map $d$ has image inside $\HH^2\subset \CC\PP^1$ (up to conjugation in $\psl(2,\CC)$). Here we see $\CC\PP^1$ as the one-point compactification of $\CC$, and $\HH^2$ as the open upper half-plane inside $\CC$. Since the image of the developing map is preserved by $\rho$, branched hyperbolic structures have real holonomy $\rho:\pi_1(M)\to\psl(2,\RR)\subset\psl(2,\CC)$. The condition of having real holonomy is, however, not enough to guarantee that the structure is hyperbolic, as there exist (even unbranched) projective structures with real holonomy which are not hyperbolic (cf.\ for example \cite{hejhal:1975aa}).
	
	The condition for the holonomy to be real can be written in gauge theoretic terms related to the flat bundle $(E,\nabla_{\hbar,R})$. Recall that the holonomy of this connection lifts to $\sl$ the holonomy of the projective structure $\mathcal{P}^{\alpha}_{\beta}(\hbar, R)$. It is real if, up to gauge equivalence, it lies in $\ssl(2,\RR)$, and this happens if $\nabla_{\hbar,R}$ preserves a real structure $\tau$.
	A real structure $\tau$ is a $\CC$-antilinear automorphism of the bundle $E$ such that $\tau^2=\mathrm{Id}_E$. The condition that $\nabla_{\hbar,R}$ preserves $\tau$ means that $\nabla_{\hbar,R} \circ \tau = \tau_{T^*M} \circ \nabla_{\hbar,R}$, where $\tau_{T^*M}$ just acts as $\tau$ on the section part, and as complex conjugation, mapping $K \to \conj{K}$, on the form part. Equivalently $\tau_{T^*M}$ is the tensor product of $\tau$ and the real structure on the complex cotangent bundle $T^*M$. 
	
	To write in gauge theoretic language the stronger condition that the image of $d$ must be in $\HH^2$, we note that this is the same as asking for the image of $d$ to avoid the real locus $\RR\PP^1\subset \CC\PP^1$ which is the fixed point set of the involution $z\mapsto \conj{z}$ of $\CC\PP^1$.
	This translates to the condition that the line bundle $L\subset E$ induced by $d$ should avoid the fixed point locus of $\tau$ in $E$, i.e., the intersection should only be the zero section of $L$ (cf.\ \cite{alessandrini:2019aa}).
	Using this description we have the following.
	
	\begin{proposition}
		Let $\deg(L)>0$ and $|\hbar|^2R^2=1$. Then $\mathcal{P}^{\alpha}_{\beta}(\hbar, R)$ is a branched hyperbolic structure.
	\end{proposition}
	\begin{proof}
		The real structure $\tau$ defining the $\slr$-structure in $E=L\oplus L^{-1}$ is given in a holomorphic frame by $\tau(v)=C\conj{v}$ with $C=\begin{psmallmatrix}
			0 & h_R^{-1}\\
			h_R & 0 
		\end{psmallmatrix}$. It is preserved by $\nabla_{\hbar, R}=d+B$ if and only if $\nabla_{\hbar,R} \circ \tau = \tau_{T^*M} \circ \nabla_{\hbar,R}$ which in this frame reads $dC + BC = C\conj{B}$. Using expression (\ref{R family}), one concludes that this happens if and only if
		\begin{align}
			\hbar^{-1}\alpha h_R+ \hbar R^2 \conj{\beta} h_R  &= \conj{\hbar^{-1} \beta h_R^{-1}+ \hbar R^2 \conj{\alpha}h_R}\notag
			\\
			\Longleftrightarrow (\hbar^{-1}-\conj{\hbar} R^2) \alpha h_R &= (\hbar^{-1}-\conj{\hbar} R^2) \conj{\beta} h_R. \label{preservation condition}
		\end{align}
		In particular, if $|\hbar|^2R^2=1$ the equality is valid. Thus the holonomy of $\nabla_{\hbar, R}$ is real.
		To check that $\mathcal{P}^{\alpha}_{\beta}(\hbar, R)$ is branched hyperbolic we simply note that $L$ avoids the fixed locus of $\tau$. This happens since vectors fixed by the real structure, i.e., such that $\tau(v)=v$, are of the form $ \begin{psmallmatrix}
			v_1\\
			h v_1
		\end{psmallmatrix}$  and the vectors in $L$ are multiples of $ \begin{psmallmatrix}
			1\\
			0
		\end{psmallmatrix}.$
	\end{proof}

	\begin{remark}
		In the particular case $\hbar=1$ and $R=1$, this result recovers the branched hyperbolic structures constructed \cite{biswas:2021aa}. 
	\end{remark}

	\subsection{Deformations of geometric structures} 
	\label{section deformation of Hitchin components}
	
	Using the previous results one can interpret our construction in terms of geometric structures on $X$ as follows. Start by fixing $\hbar$ with $|\hbar|=1$.
	Then the branched projective structure $\mathcal{P}^{\alpha}_\beta(1,R)$ interpolates between a branched hyperbolic structure, at $R=1$, and a partial oper, i.e., a complex projective compatible with $X$, at $R=0$. In the specific case of $\hbar=1$, the branched hyperbolic structure is exactly the one coming from the non-abelian Hodge correspondence.

	In particular, take any Higgs bundle in a Hitchin component, i.e., one of the form 
	$$
	\left(E=K^{1/2}\oplus K^{-1/2}, \Phi=\begin{pmatrix}
		0 & q \\
		1 & 0
	\end{pmatrix}\right),
	$$ 
	where $q\in \holosecs{K^{2}}$ is a quadratic differential. In the previous notation $\alpha=q$ and $\beta=1$, and the condition $\Div(\alpha)\geq\Div(\beta)$ is satisfied, so our construction goes through.
	The non-abelian Hodge correspondence produces a connection which is the holonomy of the (unbranched) Fuchsian hyperbolic structure $\mathcal{P}^q_1(1,R=1)$.
	Then, for $\hbar=1$, and by varying $R$, we obtain the curve $\mathcal{P}^q_1(1,R)$ of complex projective structures. This curve interpolates between this Fuchsian structure, at $R=1$, and the structure of oper determined by $q$, at $R=0$, i.e., the complex projective structure compatible with $X$ and which is classically obtained from the solutions of the Schwarz equation with quadratic differential precisely equal to $q$.
	
	\subsection{Further questions} \label{section further questions}
	
	The fact that the connections $\nabla_{\hbar,R}$ in the family (\ref{R family}) are associated with branched projective structures comes with the restriction $\Div(\alpha)\geq \Div(\beta)$. 
	
	The description of the set of Higgs bundles $\Phi$ for which this holds is determined in general by Brill--Noether type considerations. Only if $(L \oplus L^{-1}, \Phi)$ lies in a Hitchin component the construction works for all $\Phi$ and the structures so obtained are unbranched (cf.\ Example~\ref{example hitchin} and also Remark~\ref{example hitchin component}). 
	Note that this restriction cannot be lifted, since then the differential $\mu=\hbar^2 R^2 \frac{\conj{\alpha}}{\beta} h_R^2$ will have points where $|\mu|=1$, and it will no longer be a Beltrami differential. Another way to see this problem is precisely that the smooth second fundamental form of $L$, given by $\hbar^{-1} \beta + \hbar R^2 \conj{\alpha}h_R^2$, will have zeros at other points, besides the zeros of $\alpha$ and $\beta$, namely over the real subvariety of $M$ of points where $\beta(z)=\hbar^2 R^2 \conj{\alpha}(z)h_R^2(z)$. 
	One might ask what kind of structures exist that relate to the representations in this case, for which the restriction does not hold. Equivalently, what is the correct class of structures which are uniformized by $\slr$-Higgs bundles in general? Is the family appearing in the conformal limit kept within that class? This seems a plausible framework since the conformal limit exists regardless of the restrictions on $\alpha$ and $\beta$.
	
	It is well known that the fiber of the holonomy map $\mathscr{P}(M)\to \mathcal{M}^{G}_{\mathrm{B}}$ is infinite and discrete (see, e.g., Dumas~\cite{dumas}). It would thus be interesting to study the ambiguity in the projective structures $\mathcal{P}^{\alpha}_\beta(\hbar,R)$ mapping to the same holonomy when varying $\alpha$, $\beta$, $\hbar$ and $R$. 
	
	\singlespacing

	\newcommand{\etalchar}[1]{$^{#1}$}


\begin{thebibliography}{DFK{\etalchar{+}}21}
		
		\bibitem[ADL21]{alessandrini-etal:2021}
		Daniele Alessandrini, Colin Davalo, and Qiongling Li.
		\newblock Projective structures with (quasi-)Hitchin holonomy, 2021.
		\newblock \texttt{2110.15407}.
		
		\bibitem[Ale19]{alessandrini:2019aa}
		Daniele Alessandrini.
		\newblock Higgs bundles and geometric structures on manifolds.
		\newblock {\em SIGMA Symmetry Integrability Geom. Methods Appl.}, 15:Paper 039,
		32, 2019.
		
		\bibitem[Bar10]{baraglia:2010}
		David Baraglia.
		\newblock {\em G2 geometry and integrable systems}.
		\newblock PhD thesis, University of Oxford, 2010.
		\newblock \texttt{arxiv:1002.1767}.
		
		\bibitem[BBDH21]{biswas:2021aa}
		Indranil Biswas, Steven Bradlow, Sorin Dumitrescu, and Sebastian Heller.
		\newblock Uniformization of branched surfaces and {H}iggs bundles.
		\newblock {\em Internat. J. Math.}, 32(13):Paper No. 2150096, 19, 2021.
		
		\bibitem[BD05]{beilinson:aa}
		Alexander Beilinson and Vladimir Drinfeld.
		\newblock Opers, 2005.
		\newblock \texttt{arxiv:math/0501398}.
		
		\bibitem[BDH23]{biswas:2023aa}
		Indranil Biswas, Sorin Dumitrescu, and Sebastian Heller.
		\newblock Branched {${\rm SL}(r,\Bbb{C})$}-opers.
		\newblock {\em Int. Math. Res. Not. IMRN}, (10):8311--8355, 2023.
		
		\bibitem[Col20]{collier:2020}
		Brian Collier.
		\newblock {$\mathsf{SO}(n, n+1)$}-surface group representations and {H}iggs
		bundles.
		\newblock {\em Ann. Sci. \'{E}c. Norm. Sup\'{e}r. (4)}, 53(6):1561--1616, 2020.
		
		\bibitem[CT23]{collier:2023}
		Brian Collier and Jérémy Toulisse.
		\newblock Holomorphic curves in the 6-pseudosphere and cyclic surfaces, 2023.
		\newblock \texttt{arxiv:2302.11516}.
		
		\bibitem[CW19]{collier:2018aa}
		Brian Collier and Richard Wentworth.
		\newblock Conformal limits and the {B}ia\l ynicki-{B}irula stratification of
		the space of {$\lambda$}-connections.
		\newblock {\em Adv. Math.}, 350:1193--1225, 2019.
		
		\bibitem[DFK{\etalchar{+}}21]{dumitrescu:2021uz}
		Olivia Dumitrescu, Laura Fredrickson, Georgios Kydonakis, Rafe Mazzeo, Motohico
		Mulase, and Andrew Neitzke.
		\newblock From the {H}itchin section to opers through nonabelian {H}odge.
		\newblock {\em J. Differential Geom.}, 117(2):223--253, 2021.
		
		\bibitem[DS85]{drinfeld:1985aa}
		V.~G. Drinfel'd and V.~V. Sokolov.
		\newblock Lie algebras and equations of Korteweg-de Vries type.
		\newblock {\em Journal of Soviet Mathematics}, 30(2):1975--2036, jul 1985.
		
		\bibitem[Dum09]{dumas}
		David Dumas.
		\newblock Complex projective structures.
		\newblock In {\em Handbook of {T}eichm\"{u}ller theory. {V}ol. {II}}, volume~13
		of {\em IRMA Lect. Math. Theor. Phys.}, pages 455--508. Eur. Math. Soc.,
		Z\"{u}rich, 2009.
		
		\bibitem[DW07]{daskalopoulos:2007wy}
		Georgios~D. Daskalopoulos and Richard~A. Wentworth.
		\newblock Harmonic maps and {T}eichm\"{u}ller theory.
		\newblock In {\em Handbook of {T}eichm\"{u}ller theory. {V}ol. {I}}, volume~11
		of {\em IRMA Lect. Math. Theor. Phys.}, pages 33--109. Eur. Math. Soc.,
		Z\"{u}rich, 2007.
		
		\bibitem[Gai14]{gaiotto:2014td}
		Davide Gaiotto.
		\newblock Opers and \text{TBA}, 03 2014.
		\newblock \texttt{arxiv:1403.6137}.
		
		\bibitem[GKM00]{gallo:2000aa}
		Daniel Gallo, Michael Kapovich, and Albert Marden.
		\newblock The monodromy groups of {S}chwarzian equations on closed {R}iemann
		surfaces.
		\newblock {\em Ann. of Math. (2)}, 151(2):625--704, 2000.
		
		\bibitem[GT01]{gilbarg:2001aa}
		David Gilbarg and Neil~S. Trudinger.
		\newblock {\em Elliptic partial differential equations of second order}.
		\newblock Classics in Mathematics. Springer-Verlag, Berlin, 2001.
		\newblock Reprint of the 1998 edition.
		
		\bibitem[Gun67]{gunning:1967aa}
		R.~C. Gunning.
		\newblock Special coordinate coverings of {R}iemann surfaces.
		\newblock {\em Math. Ann.}, 170:67--86, 1967.
		
		\bibitem[Hej75]{hejhal:1975aa}
		Dennis~A. Hejhal.
		\newblock Monodromy groups and linearly polymorphic functions.
		\newblock {\em Acta Math.}, 135(1):1--55, 1975.
		
		\bibitem[Hit87]{hitchin:1987vg}
		N.~J. Hitchin.
		\newblock The self-duality equations on a {R}iemann surface.
		\newblock {\em Proc. London Math. Soc. (3)}, 55(1):59--126, 1987.
		
		\bibitem[Hub06]{hubbard:2006ve}
		John~Hamal Hubbard.
		\newblock {\em Teichm\"{u}ller theory and applications to geometry, topology,
			and dynamics. {V}ol. 1}.
		\newblock Matrix Editions, Ithaca, NY, 2006.
		\newblock Teichm\"{u}ller theory, With contributions by Adrien Douady, William
		Dunbar, Roland Roeder, Sylvain Bonnot, David Brown, Allen Hatcher, Chris
		Hruska and Sudeb Mitra, With forewords by William Thurston and Clifford
		Earle.
		
		\bibitem[Lab07]{labourie:2007}
		Fran\c{c}ois Labourie.
		\newblock Flat projective structures on surfaces and cubic holomorphic
		differentials.
		\newblock {\em Pure Appl. Math. Q.}, 3(4):1057--1099, 2007.
		
		\bibitem[Leh87]{lehto:1987aa}
		Olli Lehto.
		\newblock {\em Univalent functions and {T}eichm\"{u}ller spaces}, volume 109 of
		{\em Graduate Texts in Mathematics}.
		\newblock Springer-Verlag, New York, 1987.
		
		\bibitem[Li19]{li:2019aa}
		Qiongling Li.
		\newblock An introduction to {H}iggs bundles via harmonic maps.
		\newblock {\em SIGMA Symmetry Integrability Geom. Methods Appl.}, 15:Paper No.
		035, 30, 2019.
		
		\bibitem[Man72]{mandelbaum:1972aa}
		Richard Mandelbaum.
		\newblock Branched structures on {R}iemann surfaces.
		\newblock {\em Trans. Amer. Math. Soc.}, 163:261--275, 1972.
		
		\bibitem[Man73]{mandelbaum:1973aa}
		Richard Mandelbaum.
		\newblock Branched structures and affine and projective bundles on {R}iemann
		surfaces.
		\newblock {\em Trans. Amer. Math. Soc.}, 183:37--58, 1973.
		
		\bibitem[Sim10]{simpson:2010aa}
		Carlos Simpson.
		\newblock Iterated destabilizing modifications for vector bundles with
		connection.
		\newblock In {\em Vector bundles and complex geometry}, volume 522 of {\em
			Contemp. Math.}, pages 183--206. Amer. Math. Soc., Providence, RI, 2010.
		
		\bibitem[Thu97]{thurston:1997aa}
		William~P. Thurston.
		\newblock {\em Three-dimensional geometry and topology. {V}ol. 1}, volume~35 of
		{\em Princeton Mathematical Series}.
		\newblock Princeton University Press, Princeton, NJ, 1997.
		
		\bibitem[Thu22]{thurston:2022aa}
		William~P. Thurston.
		\newblock {\em The geometry and topology of three-manifolds. {V}ol. {IV}}.
		\newblock American Mathematical Society, Providence, RI, 2022.
		\newblock Edited and with a preface by Steven P. Kerckhoff and a chapter by J.
		W. Milnor.
		
	\end{thebibliography}
\end{document}